\documentclass[10pt,leqno,twoside]{amsart}

\usepackage{tikz, subfigure, xcolor} 
\usetikzlibrary{snakes}
\usepackage{amsmath}
\usepackage{amsthm}
\usepackage{amssymb}
\usepackage{amsfonts}
\usepackage{cite}
\usepackage{dsfont}
\usepackage{hyperref}
\usepackage{doi}
\usepackage{comment}

\newtheorem{theorem}{Theorem}[section]
\newtheorem{lemma}[theorem]{Lemma}
\newtheorem{proposition}[theorem]{Proposition}
\newtheorem{corollary}[theorem]{Corollary}
\theoremstyle{definition}

\newtheorem{remark}[theorem]{Remark}
\newtheorem{remarks}[theorem]{Remarks}

\numberwithin{equation}{section}

\renewcommand{\div}{\mathrm{div}\,}    %div anstatt geteilt

\newcommand{\lpso}{L^p_{\overline{\sigma}}(\Omega)}

\providecommand{\norm}[1]{\lVert#1\rVert} %Norm
\providecommand{\abs}[1]{\lvert#1\rvert} % absolut value
\newcommand{\R}{\mathbb{R}}

\newcommand{\Z}{\mathbb{Z}}
\newcommand{\N}{\mathbb{N}}

\newcommand{\cT}{{\mathcal T}}

\usepackage{eucal}

\newcommand\restr[2]{{% we make the whole thing an ordinary symbol
  \left.\kern-\nulldelimiterspace % automatically resize the bar with \right
  #1 % the function
  \vphantom{\big|} % pretend it's a little taller at normal size
  \right|_{#2} % this is the delimiter
  }}

\title{Analyticity of solutions to the primitive equations}

\keywords{Maximal $L^q$ regularity, Primitive Equations, Global strong well-posedness, regularity of solutions\\
This work was partly supported by the DFG International Research Training Group IRTG 1529 and the JSPS Japanese-German Graduate Externship on Mathematical Fluid Dynamics. 
The first author is partly supported by JSPS through grant Kiban S (No. 26220702), Kiban A (No. 17H01091), Kiban B (No. 16H03948) and the second and fourth author are supported by IRTG 1529 at TU Darmstadt.}

\author[Giga]{Yoshikazu Giga} 

\address{Graduate School of Mathematical Sciences, University of Tokyo, Komaba 3-8-1, Meguro-ku, Tokyo, 153-8914, Japan }
\email{labgiga@ms.u-tokyo.ac.jp}

\author[Gries]{Mathis Gries} 

\address{Departement of Mathematics,
	TU Darmstadt, Schlossgartenstr. 7, 64289 Darmstadt, Germany}
\email{gries@mathematik.tu-darmstadt.de}

\author[Hieber]{Matthias Hieber} 
\address{Departement of Mathematics,
	TU Darmstadt, Schlossgartenstr. 7, 64289 Darmstadt, Germany}
\email{hieber@mathematik.tu-darmstadt.de}

\author[Hussein]{Amru Hussein} 
\address{Departement of Mathematics,
	TU Darmstadt, Schlossgartenstr. 7, 64289 Darmstadt, Germany}
\email{hussein@mathematik.tu-darmstadt.de}

\author[Kashiwabara]{Takahito Kashiwabara}
\address{Graduate School of Mathematical Sciences, The University of Tokyo, 3-8-1 Komaba, Meguro, Tokyo 153-8914, Japan}
\email{tkashiwa@ms.u-tokyo.ac.jp}

\begin{document}

\begin{abstract}
\noindent
This article presents the maximal regularity approach to the primitive equations. It is proved that the $3D$ primitive equations on cylindrical domains 
admit a unique, global strong solution for initial data lying in  the critical solonoidal Besov space $B^{2/p}_{pq}$ for $p,q\in (1,\infty)$ with $1/p+1/q \leq 1$. This solution
regularize instantaneously and  becomes even real analytic for $t>0$. 
%The above condition on the initial data generalizes in particular the  known results on admissible initial data for strong solutions since $H^{2/p,p}\subset B^{2/p}_{pq}$.
\end{abstract}

\maketitle

\section{Introduction}

The primitive equations for the ocean and atmosphere are considered to be a fundamental model for geophysical flows which is derived from Navier-Stokes equations 
assuming a hydrostatic balance for the pressure term in the vertical direction.  
The mathematical analysis of the primitive equations commenced by Lions, Teman and Wang in the series of articles \cite{Lionsetall1992, Lionsetall1992_b, Lionsetall1993}; for a survey of 
known results and further references to the literature, we refer to the recent article by  Li and Titi \cite{LiTiti2016}. 

In contrast to Navier-Stokes equations, the $3D$ primitive equations admit a unique, global, strong solution for arbitrary large data in $H^1$. This breakthrough result was proved 
by Cao and Titi \cite{CaoTiti2007} in 2007 using energy methods.  A different approach to the  primitive equations, based on methods of evolution equations, has been presented  
in \cite{HieberKashiwabara2015}. There a Fujita-Kato type iteration scheme was developed in addition to  $H^2$-{a priori} bounds for the solution. 

It is the aim of this article to present a third approach to the primitive equations, this time based on techniques from the theory of maximal $L^q$-regularity. This approach has 
several advantages compared to the two other approaches.  Let us note first that combining this approach with the so-called parameter-trick due to Angenent \cite{Angenent1990_2, Angenent1990}  
we are able to rigorously prove the immediate smoothing effect of solution. In particular, the solution regularizes instantaneously to become 
{\em real analytic in time and space}, a property,  which is interesting for its own sake. 

Real analyticity of solutions of certain classes of partial differential equations is 
usually difficult to prove directly, however, our approach allows to employ the implicit function theorem and thus yields an elegant strategy  for solving this problem.    
For first results in this direction concerning  the Navier-Stokes equations, we refer  to the work of Masuda \cite{Masuda1980}; for the general theory  concerning quasilinear systems 
and refinements, we refer to \cite{PruessSimonett2016}. Let us remark that $C^\infty$-smoothness properties  of the solutions to the primitive equations have also been obtained by 
Li and Titi in \cite{LiTiti2015} by very different methods. 

The above  regularizing effect plays an important role when extending local solutions to global ones by means of certain {a priori} bounds. So far, in order to control the 
existence time in $L^p$-spaces, $H^2$-{a priori} bounds have been used in \cite{HieberKashiwabara2015, HieberKashiwabaraHussein2016}. In the following, we show that {a priori} bounds 
in the maximal regularity space $L^2(0,T;H^2)\cap H^1(0,T;L^2)$ are already sufficient to prove the global existence of a solution in  $L^q$-$L^p$-spaces. Smoothing properties of the solution 
play also a very  important  role in the proof of the recent results on the existence of global, strong  solutions to the primitive equations for rough initial data lying in the 
anisotropic and scaling invariant spaces $L^{\infty}(L^1)$ and $L^{\infty}(L^p)$; see \cite{GigaGriesHusseinHieberKashiwabara2017NN}.

Second, our approach allows to prove the existence and uniqueness of a global, strong  solution for initial values lying in critical spaces, which in the given situation are the  Besov spaces 
\begin{align*}
B^{\mu}_{pq} \quad \hbox{for} \quad p,q\in (1,\infty)\quad \hbox{with} \quad 1/p+1/q\leq \mu \leq 1. 
\end{align*}
Here, we use in an essential way the concept of time weights for maximal $L^p$-regularity, see \cite{PruessSimonett2016} and  \cite{PruessWilke2016}. The above spaces seem to be the 
largest spaces of initial data for which one obtains the existence of a unique, global strong solution to the primitive equations when considering the problem within the 
$L^q-L^p$-framework with $1<p,q<\infty$.      

Note that the above  spaces are critical function spaces, where by critical is understood in the sense discussed e.g. in \cite{PruessSimonettWilke2017}. These spaces correspond in the situation 
of the Navier-Stokes equations to the critical function spaces $B^{n/p-1}_{pq}$ introduced by Cannone \cite{Can97} for the full space  case $\R^n$, and by Pr\"uss and Wilke 
\cite{PruessWilke2016} for bounded domains. Also, for other solution classes for the $3$-d Navier-Stokes equations initial conditions in Besov spaces occur such as in the works by Farwig, Giga and Hsu \cite{Farwigetall2016, Farwigetall2017, Farwigetall2017b} which investigate solutions which are continuous in time and taking values in the class of Besov spaces $B_{p,q}^{3/p-1}$ for suitable coefficients $p,q$ including the case $q=\infty$.

%Also, solutions to the $3$-d Navier-Stokes 
%equations, which are continuous in time and taking values in the class of Besov spaces $B_{p,q}^{3/p-1}$ for suitable coefficients $p,q$ including the case $q=\infty$ have been investigated by  Farwig, Giga and Hsu \cite{Farwigetall2016, Farwigetall2017}. PREPRINT

%There is a series of works on the Navier-Stokes 
%equations in bounded domains with initial data in Besov spaces initiated by 
%Farwig, Sohr and Varnhorn in \cite{Farwigetall2015}
%and extended to arbitrary exponents including 
%$\infty$ by 

Choosing in particular $p=q=2$ and $\mu=1$ and noting that $B^{1}_{22}=H^1$, we rediscover in particular the celebrated result by Cao and Titi \cite{CaoTiti2007}.
Furthermore, choosing  $p,q>2$ allows us to enlarge the space of admissible initial values $H^{2/p,p}$ as constructed in \cite{HieberKashiwabara2015, HieberKashiwabaraHussein2016}
to the above more general Besov space setting since $H^{2/p,p} \subset B^{2/p}_{pq}$. This class of initial values is also used in our works 
\cite{GigaGriesHusseinHieberKashiwabara2017DN} on rough initial data lying in $L^{\infty}(L^p)$ for the case of mixed Dirichlet and Neumann boundary conditions. There, the solutions obtained in this article  serve as reference solutions belonging to initial values in 
$B^{\mu}_{pq}$ and where parameters are chosen in such a way  that $B^{\mu}_{pq}\hookrightarrow C^1$.

Third, our approach allows us to treat various types boundary conditions, such as  Dirichlet, Neumann, and mixed Dirichlet and Neumann boundary conditions in a unified way. The corresponding 
$L^2(0,T;H^2)\cap H^1(0,T;L^2)$ a priori bounds  for the case of mixed Dirichlet and Neumann boundary conditions rely on results obtained in 
\cite{HieberKashiwabara2015, GaldiHieberKashiwabara2015} and for Dirichlet boundary conditions these bounds can be obtained similarly.  
For the remaining case of pure Neumann boundary conditions, we present a proof of these bounds in Section~\ref{sec:proofs}.  We hence obtain the existence of a unique, global, strong solution to the 
primitive equations for all of these boundary conditions.

This article is organized as follows: In Section~\ref{sec:pre} we describe the setting in detail and the main results are presented in Section~\ref{sec:main}. 
Some information to the linear theory is recapped and supplemented in Section~\ref{sec:lin}. We collect the  relevant results on maximal $L^q$-regularity in Section~\ref{sec:maxreg}; they will be then applied in the  proofs of our  main results given in Section~\ref{sec:proofs}. Finally, the various approaches are compared  in Section~\ref{sec:concludingremark}.

{\bf Acknowledgment}. The authors would like to thank Jan Pr\"uss and Mathias Wilke for making their recent article \cite{PruessWilke2016} available to us prior to publication.

\section{Preliminaries}\label{sec:pre}

Consider a cylindrical domain $\Omega =  G \times (-h,0) \subset \R^3$ with $ G=(0,1)\times(0,1)$, $h>0$. For simplicity, we investigate  the primitive equations in the isothermal setting 
and denote by $v\colon\Omega \rightarrow \R^2$ the vertical velocity of the fluid and $\pi_s\colon G \rightarrow \R$ its surface pressure. There exist  several equivalent 
formulations of the primitive equations, depending on whether the horizontal velocity $w=w(v)$ is completely substituted by   the vertical velocity $v$ and the full pressure 
by the surface pressure, respectively, compare e.g. \cite{HieberKashiwabara2015}. For the purpose of this article the following representation of the primitive equations is the most 
convenient, 
\begin{align}\label{eq:primequiv}
\left\{
\begin{array}{rll}
\partial_t v + v \cdot \nabla_H v + w(v) \cdot \partial_z v - \Delta v + \nabla_H \pi_s  & = f, &\text{ in } \Omega \times (0,T),  \\
\mathrm{div}_H \overline{v} & = 0, &\text{ in } \Omega \times (0,T), \\
v(0) & = v_0, &\text{ in } \Omega, 
\end{array}\right.
\end{align}
where denoting by $x,y\in G$ horizontal coordinates and by $z\in (-h,0)$ the vertical one, we use the notations
\begin{align*}
\Delta = \partial_x^2 + \partial_y^2+ \partial_z^2, \quad \nabla_H  = (\partial_x, \partial_y)^T, \quad \div_H v= \partial_x v_1+ \partial_y v_2 \quad \hbox{and} \quad \overline{v}:=\frac{1}{h}\int_{-h}^0 v(\cdot,\cdot, \xi)d\xi.
\end{align*}
Here  the horizontal velocity $w=w(v)$ is given by
\begin{eqnarray*}
w(v)(x,y,z) = -\int_{-h}^z \mathrm{div}_H v(x,y, \xi)d\xi, \quad \hbox{where} \quad w(x,y,-h)=w(x,y,0)=0.
\end{eqnarray*}
The equations \eqref{eq:primequiv} are supplemented by the mixed boundary conditions on
\begin{align*}
\Gamma_{u} = G \times \{0\}, \quad \Gamma_b = G \times \{-h\} \quad \hbox{and} \quad \Gamma_l = \partial G \times (-h,0),
\end{align*}
i.e. the upper, bottom and lateral parts of the boundary $\partial\Omega$, respectively,
given by
\begin{align}\label{eq:bc}
v, \pi_s  \hbox{ are periodic } \hbox{on } \Gamma_l \times (0,\infty), \nonumber \\
v = 0 \hbox{ on } \Gamma_D \times (0,\infty) \quad \hbox{and} \quad \partial_z v = 0  \hbox{ on } \Gamma_N \times (0,\infty).
\end{align}
where Dirichlet, Neumann and mixed boundary conditions are comprised by the notation
\begin{align*}
\Gamma_D \in \{\emptyset, \Gamma_{u}, \Gamma_{b}, \Gamma_{u}\cup \Gamma_{b}\}\quad  \hbox{and} \quad \Gamma_N = (\Gamma_{u}\cup \Gamma_{b})\setminus \Gamma_D.
\end{align*}
In the literature several sets of boundary conditions are considered. So, in \cite[Equation (1.37) and (1.37)']{Lionsetall1992} Dirichlet and mixed Dirichlet Neumann boundary conditions 
are considered, respectively, while in \cite{CaoTiti2007} Neumann boundary conditions are assumed. 
%Here, mixed boundary conditions are studied as the most illustrative case since pure Dirichlet conditions do not add more difficulties while Neumann boundary conditions yield a simpler linearized% system. 

Similarly to the Navier-Stokes equations, one may consider \textit{hydrostatically  solenoidal vector fields} as a subspace of $L^p(\Omega)^2$ for $p\in (1,\infty)$ which, following the approach developed 
in \cite[Sections 3 and 4]{HieberKashiwabara2015} is defined by 
\begin{align}\label{eq:lpso}
L^p_{\overline{\sigma}}(\Omega) &= \overline{\{v\in C^{\infty}_{per}(\overline{\Omega})^2 : \div_H \overline{v} = 0\}}^{\norm{\cdot}_{L^{p}(\Omega)^2}},
\end{align}
where horizontal periodicity is modeled by the function spaces $C^{\infty}_{per}(\overline{\Omega})$ and $C^{\infty}_{per}(\overline{G})$ is defined as in 
\cite[Section 2]{HieberKashiwabara2015}, where smooth functions are periodic only with respect to $x,y$ coordinates and not necessarily in the $z$ coordinate.

Furthermore, there exists a continuous projection $P_p$, called the \textit{hydrostatic Helmholtz projection}, from $L^p(\Omega)^2$ onto $\lpso$, see \cite{GigaGriesHusseinHieberKashiwabara2016} and \cite{HieberKashiwabara2015} for details. In particular, $P_p$ annihilates the pressure term $\nabla_H\pi_s$.

For $p\in(1,\infty)$ and $s\in [0,\infty)$ define the spaces
\begin{align*}
H^{s,p}_{per}(\Omega) := \overline{C^{\infty}_{per}(\overline{\Omega})}^{\norm{\cdot}_{H^{s,p}(\Omega)}} \quad \hbox{and} \quad
H^{s,p}_{per}(G) := \overline{C^{\infty}_{per}(\overline{G})}^{\norm{\cdot}_{H^{s,p}(G)}},
\end{align*}
where $H^{0,p}_{per}:= L^p$. Here $H^{s,p}(\Omega)$ denotes the Bessel potential spaces, which are defined as restrictions of Bessel potential spaces on the  whole space to $\Omega$, 
compare e.g \cite[Definition 3.2.2.]{Triebel}. It is well known that the space $H^{s,p}(\Omega)$ coincides with the  classical Sobolev space $W^{m,p}(\Omega)$ provided $s=m \in \N$.

Also, we define for $p,q\in(1,\infty)$ and $s\in [0,\infty)$ the Besov spaces
\begin{align*}
B^{s}_{pq,per}(\Omega) := \overline{C^{\infty}_{per}(\overline{\Omega})}^{\norm{\cdot}_{B^{s}_{pq}(\Omega)}} \quad \hbox{and} \quad
B^{s}_{pq,per}(G) := \overline{C^{\infty}_{per}(\overline{G})}^{\norm{\cdot}_{B^{s}_{pq}(G)}},
\end{align*}
where $B^{s}_{pq}$ denotes Besov spaces, which are defined as restrictions of Besov spaces on the whole space $B^{s}_{p,q}(\R^3)$, compare e.g. \cite[Definitions 3.2.2]{Triebel}.

Following \cite{HieberKashiwabara2015}, we define the \textit{hydrostatic Stokes operator} $A_p$ in $\lpso$ as
\begin{align*}
A_p v := P_p \Delta v, \quad D(A_p) :=  \{v\in H_{per}^{2,p}(\Omega)^2 : \restr{\partial_z v}{\Gamma_N} = 0, \,\restr{v}{\Gamma_D} = 0 \} \cap L^p_{\overline{\sigma}}(\Omega).
\end{align*}

In \cite{GigaGriesHusseinHieberKashiwabara2016} it has been shown that $A_p$ has the property of maximal $L^q$-regularity. Following \cite[Theorem 3.2]{PruessSimonett2004} this is 
equivalent to maximal $L^q$-regularity of $A_p$ in time-weighted spaces, which are defined for $\mu \in (1/q,1]$ and for $k\in \N$ recursively by 
\begin{align*}
L^{q}_{\mu}(J; D(A_p)) &= \{ v\in L^{1}_{loc}(J; D(A_p)) \colon t^{1-\mu} v \in L^q(J; D(A_p))\},\\
H^{1,q}_{\mu}(J; \lpso) &= \{ v\in L^{q}_{\mu}(J; \lpso) \cap H^{1,1}(J; \lpso) \colon t^{1-\mu} v_t \in L^q(J; \lpso)\},\\
\vdots
\\
H^{k+1,q}_{\mu}(J; \lpso) &= \{ v\in H^{k,q}_{\mu}(J; \lpso) \colon v_t \in H^{k,q}(J; \lpso)\}.
\end{align*}
Here $v_t$ stands for the time derivative of $v$ %in the distributional sense 
and $J=(0,T)$ denotes for  $0<T\leq \infty$ a time interval.

The natural trace spaces of these spaces are determined  by real interpolation $(\cdot, \cdot)_{\theta,q}$ for $\theta \in (0,1)$ and $p,q\in (1,\infty)$. They can be  computed explicitly 
in terms of Besov spaces, as we will prove  in Section~\ref{sec:lin}.

\begin{lemma}\label{lem:real}
	Let $\theta \in (0,1)$ and $p,q\in (1,\infty)$. Then for $X_{\theta,q}:= (\lpso, D(A_p))_{\theta,q}$ it holds that
	\begin{align*}
	X_{\theta,q} =\begin{cases}
	\{ v\in  B^{2\theta}_{pq, per}(\Omega) \cap \lpso \colon \restr{\partial_z v}{\Gamma_N} = 0, \,\restr{v}{\Gamma_D} = 0 \}, & \tfrac{1}{2} + \tfrac{1}{2p} < \theta < 1, \\
	\{ v\in  B^{2\theta}_{pq, per}(\Omega) \cap \lpso \colon  \restr{v}{\Gamma_D} = 0 \}, & \tfrac{1}{2p} < \theta < \tfrac{1}{2} + \tfrac{1}{2p}, \\ 
	B^{2\theta}_{pq, per}(\Omega) \cap \lpso, & 0 < \theta < \tfrac{1}{2p}. 
	\end{cases}
	\end{align*}
\end{lemma}

\section{Main Results}\label{sec:main} 

Our first main result is the global, strong well-posedness of the primitive equations for arbitrarily large data in the critical Besov spaces defined above in Lemma \ref{lem:real}.   

\begin{theorem}[Global well-posedness]\label{thm:glob}
Let $p,q\in (1,\infty)$ such that $1/p + 1/q\leq 1$. For $0<T<\infty$ let $ \mu \in  [1/p + 1/q,1]$,
    \begin{align*}
    v_0\in X_{\mu-1/q,q} \qquad \hbox{and} \qquad P_p f \in H^{1,q}_{\mu}(0,T; \lpso) \cap  H^{1,2}(\delta,T; L^{2}_{\overline{\sigma}}(\Omega)) 
    \end{align*}	
  %  \begin{align*}
  %  v_0\in X_{\mu-1/q,q} \qquad \hbox{and} \qquad P_p f \in L^{q}_{\mu}(0,T; \lpso) \cap  H^{1,2}(\delta,T; L^{2}_{\overline{\sigma}}(\Omega)) 
  %  \end{align*}
	for some $\delta>0$ sufficiently small. 
	
	Then there exists a unique, strong solution $v$ to the primitive equations~\eqref{eq:primequiv} satisfying 
	\begin{align*}
	v\in H^{1,q}_{\mu}(0,T;\lpso) \cap L^{q}_{\mu}(0,T;D(A_p)). 
	\end{align*}
%	In addition one has
%	\begin{align*}
	%	t \cdot v^{\prime} \in H^{1,q}_{\mu}((0,T),\lpso) \cap L^{q}_{\mu}((0,T),D(A_p)). 
% %	\end{align*}
\end{theorem}

Considering in particular the case $p=q=2$, we are not only reproducing the known global existence result \cite{CaoTiti2007}, but state furthermore that additional time regularity of the 
forcing term improves the regularity of these solutions.

\begin{proposition}\label{thm:globl2}
Let $0<T<\infty$ and $v_0\in \{H^{1} \cap L^{2}_{\overline{\sigma}}(\Omega) \colon \restr{v}{\Gamma_D} = 0\}$.
\begin{itemize}
	\item[(a)] If $P_2 f \in L^2(0,T; L^{2}_{\overline{\sigma}}(\Omega))$, then there exists a  unique, strong solution $v$ to the primitive equations~\eqref{eq:primequiv} in
	\begin{align*}
	v\in H^{1}(0,T;L^{2}_{\overline{\sigma}}(\Omega))) \cap L^{2}(0,T;D(A_2)). 
	\end{align*}
	\item[(b)] If in addition $t\mapsto t\cdot P_2 f_t(t) \in L^2(0,T; L^{2}_{\overline{\sigma}}(\Omega))$, then 
	\begin{align*}
	t\cdot v_t\in H^{1}(0,T;L^{2}_{\overline{\sigma}}(\Omega))) \cap L^{2}(0,T;D(A_2)). 
	\end{align*}
\end{itemize}
\end{proposition}
%In particular one has that $v\in H^{1}(\delta,T;D(A_2))\hookrightarrow C([\delta,T];D(A_2))$ for $\delta\in (0,T)$. 
%To prove (a) an \textit{a priori} bound in the maximal regularity space $H^{1}(0,T;L^{2}_{\overline{\sigma}}(\Omega))) \cap L^{2}(0,T;D(A_2))$ is needed, while (b) follows from (a) by the implicit function theorem without using additional \textit{a priori} bounds, see Lemma~\ref{lemma:vt} below.
%

%n o-blow up result Theorem~\ref{thm:globex}.

The following second main theorem deals with the parabolic smoothing effect and the real analyticity of the solution. Note that the additional regularity assumption on $f$ is
needed, together with Proposition~\ref{thm:globl2}, to prove the global existence of the solution in Theorem~\ref{thm:glob} above. We set $v^{(j)}:= \partial_t^{j} v$ and denote by 
$C^{\omega}$ the space of real analytic functions.

%for the primitive equations in combination with the application of the abstract theory for maximal $L^q$-regularity taken mainly from \cite{PruessWilke2016} and \cite{PruessSimonett2016}, discussed here in Section~\ref{sec:maxreg}.  %. The necessary results will be discussed in the following section.  
 % using the abstract 

\begin{theorem}[Regularity]\label{thm:time}
	Let $v\in H^{1,q}_{\mu}(0,T;\lpso) \cap L^{q}_{\mu}(0,T;D(A_p))$ be the solution to the primitive equations for $v_0\in X_{\mu-1/q,q}$  and $P_pf\in L^{q}_{\mu}(0,T;\lpso)$ 
for $p,q,\mu$ as in Theorem~\ref{thm:glob}. 
	\begin{itemize}
		\item[(a)] If $P_p f \in H^{k,q}_{\mu}(0,T; \lpso)$ for $k\in \N_0= \N \cup \{0\}$,
		then  %the solution in Theorem~\ref{thm:glob}
		 for any $0<T'<T$ 
		\begin{align*}
		t^j \cdot v^{(j)} &\in H^{1,q}_{\mu}(0,T';\lpso) \cap L^{q}_{\mu}(0,T';D(A_p)), \quad j= 0, \ldots, k, \\
		v &\in  H_{loc}^{k+1,p}(0,T;\lpso) \cap H_{loc}^{k,q}(0,T;D(A_p)) \cap C^{k}((0,T);X_{1-1/q,q});
		\end{align*}
		\item[(b)] If $P_p f \in C^{\infty}((0,T); \lpso)$ or $P_p f \in C^{\omega}((0,T); \lpso)$, then 
		\begin{align*}
		v \in C^{\infty}((0,T); D(A_p)) \qquad \hbox{or} \qquad v \in C^{\omega}((0,T); D(A_p)),
		\end{align*}
		respectively;
		\item[(c)] If
		$P_p f \in C^{\infty}((0,T); C ^{\infty}_{per}(\Omega)^2)$ or $P_p f \in C^{\omega}((0,T); C ^{\omega}_{per}(\Omega)^2)$, then 
		\begin{align*}
		v \in C^{\infty}((0,T); C ^{\infty}_{per}(\Omega)^2) \qquad \hbox{or} \qquad v \in C^{\omega}((0,T); C ^{\omega}_{per}(\Omega)^2),
		\end{align*}
		respectively.
	\end{itemize}
	 
\end{theorem}

%The proof of the above regularity results relies on the implicit function theorem. % Masuda \cite{Masuda1980} was using this theorem for proving regularity results for the Navier-Stokes 
%equations and Angenent \cite{Angenent1990_2, Angenent1990} developed the nowadays called parameter trick valid for general quasilinear evolution equations. For proofs within 
%the $L^p$-setting we are working in and for further refinements of this method, we refer e.g. to  \cite{PruessSimonett2016}. 

The proof of the above regularity results relies on the implicit function theorem. 
Masuda \cite{Masuda1980} introduced first an extra parameter to prove the spatial analyticity of solutions to the Navier-Stokes equations using the implicit function theorem. Angenent \cite{Angenent1990_2, Angenent1990} systematically
developed the nowadays called parameter trick valid for general quasilinear evolution equations. For proofs within the $L^p$-setting we are working in and for further refinements of this method, we refer e.g. to \cite{PruessSimonett2016}.

\begin{remarks}\label{rem:main}
\begin{itemize}
\item[(a)] The solution class given in Theorem~\ref{thm:glob} admits the following embeddings
		\begin{align*}
		   H^{1,q}_{\mu}(0,T;\lpso) \cap L^{q}_{\mu}(0,T;D(A_p))&\hookrightarrow  C([0,T]; X_{\mu-1/q,q}),\\
		    H^{1,q}_{\mu}(0,T;\lpso) \cap L^{q}_{\mu}(0,T;D(A_p))&\hookrightarrow C((\delta,T]; X_{1-1/q,q}), \quad \delta>0.
		\end{align*}
\item[(b)] Replacing in Theorem~\ref{thm:glob} the condition on $P_p f$ by $P_p f \in L^{q}_{\mu}(0,T; \lpso)$ one still obtains local solutions, see Proposition~\ref{prop:loc} below. 
However, to extend this solution to a global one, one needs a priori bounds in $X_{\overline{\mu},q}$ with $\mu< \overline{\mu}\leq 1$, that is, a slightly more 
regular space, than the space of admissible initial values $X_{\mu,q}$, compare Theorem~\ref{thm:globex} below. The additional assumption on the time regularity of $P_p f$ in 
Theorem~\ref{thm:glob}  %has two consequences. First, it assures additional regularity of the local solutions such that $v\in C([0,T^*];X_{\mu-1/q,q})$ for $0<T^*\leq T$, and second it 
assures that this solution is also an $L^2$ solution, and by Proposition~\ref{thm:globl2} $v\in H^{1}(\delta,T;D(A_2))\hookrightarrow C([\delta,T];D(A_2))$ for $\delta\in (0,T)$.
\item[(c)] Concerning analyticity of solutions, for solutions $u$ to the Navier-Stokes equations in $\R^n$ with initial values in $u_0 \in L^n(\R^n)^n$ estimates on 
$\norm{D^{\beta} u}_{L^{q}(\R^n)}$, $\beta\in \N_0^n$, have been established in \cite{GigaSawada2003} using heat kernel estimates. However, the method used there is not applicable in the 
presence of boundaries.
\item[(d)]  The surface pressure $\pi_s$ can be reconstructed from $v$, compare \cite[Equation (6.2)]{HieberKashiwabaraHussein2016}. This gives $\pi_s \in L^{q}_{\mu}(0,T;H_{per,0}^{1,p}(G))$, where $H_{per,0}^{1,p}(G)=\{\pi_s\in H_{per}^{1,p}(G)\colon \int_G \pi_s =0 \}$.
\end{itemize}
\end{remarks}

 \section{Linear Theory}\label{sec:lin}
Consider  the linear problem
 \begin{align}\label{eq:HydrostaticStokes}
 \left\{
 \begin{array}{rll}
 \partial_t v -\Delta v +\nabla_H \pi_s = f,\\
 \div_H \overline{v} = 0, \\
 v(0)= v_0
 \end{array}\right.
 \end{align}
 subject  to boundary conditions \eqref{eq:bc}.  As described  in \cite{GigaGriesHusseinHieberKashiwabara2016}, the key observation for solving this problem is to  
solve first the equation \eqref{eq:HydrostaticStokes} for the surface pressure and to consider the hydrostatic Stokes operator defined by 
 \begin{align}\label{eq:Ap}
 A_p v = \Delta v - (\mathds{1}-P_p) (\partial_z v)\mid_{\Gamma_D}, \quad \hbox{for } v\in D(A_p).
 \end{align}
 %where for brevity
 %\begin{align*}
 %D_z v\mid_{\Gamma_D} = \tfrac{1}{b-a}\left(\gamma(b)\partial_z v\mid_{\Gamma_b} - \, \gamma(a) \partial_z %v\mid_{\Gamma_a}\right)
 %\end{align*}
 %with $c\in \{a,b\}$ one sets $\gamma(c)=1$ if $\Gamma_c \subset \Gamma_D$ and $\gamma(c)=0$ otherwise. 
This allows us to analyze the above linear problem by perturbation methods for the Laplacian. In particular, it was shown in \cite{GigaGriesHusseinHieberKashiwabara2016} that $-A_p$ admits  a 
bounded $H^{\infty}$-calculus and thus also maximal $L^q$-regularity.

The analysis in \cite{GigaGriesHusseinHieberKashiwabara2016} makes use of exponential stability of the hydrostatic semigroup generated by $A_p$. The latter  fact has been 
obtained first in \cite{HieberKashiwabara2015}.  We give here an alternative proof of this property via the positivity of the spectrum of $-A_p$ and collect in addition further spectral properties
of $A_p$. 

\begin{proposition}\label{prop:ellipticregularity}
 		Let $p\in (1,\infty)$ and $\lambda\geq 0$ if $\Gamma_D \neq \emptyset$ and $\lambda>0$ if $\Gamma_D=\emptyset$.
 	\begin{itemize}
 		\item[(a)] Let $v\in D(A_p)$ and $(A_p-\lambda) v = f$. If $P_p f\in H^{s,p}(\Omega)^2$ for $s\geq 0$, then $v\in H^{s+2,p}(\Omega)^2$ and there exists a constant $C>0$ such that  
 		\begin{align*}
 		\norm{v}_{H^{s+2,p}(\Omega)^2} \leq C \norm{f}_{H^{s,p}(\Omega)^2}.%+ \gamma_N\norm{v}_{H^{m,p}(\Omega)^2}\right)
 		\end{align*}
 		\item[(b)] The spectrum of $A_p$ is purely discrete and all the eigenfunctions of $A_p$ belong to  $C^{\infty}(\overline{\Omega})^2$. In particular, 
             the spectrum of $-A_p$ is independent of $p$, $\sigma(-A_p) \subset [0,\infty)$ and  $\sigma(-A_p) \subset [c,\infty)$ for some $c>0$ provided $\Gamma_D\neq \emptyset$.
 \item[(c)] The semigroup generated by $A_p-\lambda$ is exponentially stable.
 	\end{itemize}
   \end{proposition}

 \begin{remarks}
 \begin{itemize}
 \item[(a)] In the quantitative analysis of the primitive model the Coriolis force, which may be incorporated by replacing $f$ in \ref{eq:primequiv} by $f_0 k \times v$, 
plays an important role. For the qualitative analysis we omit this term since it is a zero order term, which can be included easily into our analysis by setting
 		\begin{align*}
 		A_p v = P_p \Delta + P_p \left(f_0k \times v\right).
 		\end{align*}
% Again, spectral properties of $A_p$ can be reduced to the $L^2$ case, since $f_0 k \times v$ is skew-symmetric.
\item[(b)] In \cite{Cotietall2015} the hydrostatic Stokes operator with Robin boundary conditions is considered. The strategy to solve first for the surface pressure can be applied 
in this case as well.
 \end{itemize}	
 \end{remarks}

 \begin{proof}[Proof of Proposition~\ref{prop:ellipticregularity}]
 Let $v\in D(A_p)\subset H^{2,p}$ and $P_p f\in H^{s,p}(\Omega)^2\subset L^p(\Omega)^2$. We then obtain 
 \begin{align*}
 (\Delta -\lambda)v= P_pf + (\mathds{1}-P_p) (\partial_z v)\mid_{\Gamma_D} \in  H^{\min\{s, 1-1/p-\epsilon\}} 
 \end{align*}
for some $\varepsilon>0$ and elliptic regularity for the Laplacian implies $v\in H^{2+\min\{s, 1-1/p-\epsilon\}}(\Omega)$. Iterating this argument, the assertion follows. 
 	
The discreetness of the spectrum follows from the compactness of the  embedding $D(A_p)\subset H^{2,p}(\Omega)\cap \lpso\hookrightarrow \lpso$. Again, by elliptic regularity of $A_p$, 
eigenfunctions of $A_p$ are in $C^{\infty}(\overline{\Omega})$ and thus independent of $p$. Hence, it is sufficient to compute the spectrum in the case $p=2$, 
where $-A_2=-P_2 \Delta$ is associated to the closed symmetric form defined by
 	 \begin{align*}
 	 \mathfrak{a}[v,v']:= \langle \nabla v, \nabla v'\rangle_{L^2(\Omega)^{2\times 2}}, \quad \hbox{where } v,v'\in \{H_{per}^1(\Omega) \cap L^2_{\overline{\sigma}}(\Omega) \colon v\vert_{\Gamma_D}=0\}.
 \end{align*}
 	 Eigenvalues are characterized by the min-max principle, and therefore enlarging the form domain to $\{v\in H_{per}^{1,2}(\Omega)^2 \colon v_{\Gamma_D} =0 \}$, we conclude that 
the eigenvalues of $-A_p$ can be inferred by those of $-\Delta$, and in the proof of $(a)$ one can choose any $\lambda\geq 0$ for $\Gamma_D\neq \emptyset$ and any $\lambda>0$ for $\Gamma_D= \emptyset$. The exponential stability of the hydrostatic semigroup follows from the fact that for analytic semigroups the growth bound coincides with the spectral bound.
 	 %defined by the corresponding form. 
 \end{proof}

\begin{proof}[Proof of Lemma~\ref{lem:real}]
Note that the cylindrical domain $\Omega$ can be extended to the $3$ dimensional full torus, and therefore we may equivalently define the function spaces appearing in Lemma \ref{lem:real} 
as restrictions of the function spaces on the full torus, compare e.g. \cite[Chapter 9]{Triebel}. Then, the real interpolation space $X_{\theta,q}$ can be computed by means 
of these retractions and co-retractions, which are related to function spaces on the full torus, compare \cite[Theorem 1.2.4]{Triebel1978} and \cite[1.7.1 Theorem 1]{Triebel1978}. 
Since $D(A_p)$ consists of functions with Dirichlet, Neumann or mixed boundary conditions one can extend these by odd and even extensions to the full torus, which defines a 
co-retraction to the full torus, where real interpolation spaces are known, compare e.g. \cite[Section 4.11.1]{Triebel1978}. Since  odd and even function have vanishing traces and 
traces of the derivatives, respectively, the traces can be found in the interpolation spaces as long as they are well-defined. Alternatively, the retractions and co-retractions 
given in \cite[Section 4]{HieberKashiwabaraHussein2016} can be used.
\end{proof}

\section{Semilinear evolution equations and maximal $L^q$-regularity}\label{sec:maxreg}
In this section we present results on semilinear evolution equations, which then will be applied to the primitive equations  in Section~\ref{sec:proofs}. We also present only a 
simplified version of a more general result for quasilinear systems due to Pr\"uss and Wilke \cite{PruessWilke2016}, which, however, will be sufficient for our purposes. 

Let $X_0,X_1$ be Banach spaces such that $X_1 \hookrightarrow X_0$ is densely embedded, and let $A\colon X_1 \rightarrow X_0$ be bounded. 
The aim is to solve the semi-linear problem for $0<T\leq \infty$
\begin{align}\label{eq:abstrpb}
u^{\prime} + Au = F(u) + f, \qquad 0 < t < T, \qquad u(0)=u_0.
\end{align}
%within the framework of Banach space valued $L^q$ spaces, $q\in (1,\infty)$, with time-weights, where $u^{\prime}$ denotes derivative with respect to time. Therefore, 
For a Banach space $X$, a time weight $\mu\in (1/q,1]$, a time interval $J\subset [0,\infty)$ and $k\in \N$, we set
\begin{align*}
L^q_{\mu}(J;X) &= \{u\in L^{1}_{loc}(J;X) \colon t^{1-\mu} u \in L^q(J;X)\}, \\
H^{1,q}_{\mu}(J;X) &= \{u\in L^{1}_{loc}(J;X) \cap H^{1,1}(J;X) \colon t^{1-\mu} u^{\prime} \in L^q(J;X)\}, \\
\vdots \\
H^{k+1,q}_{\mu}(J;X) &= \{u\in L^{1}_{loc}(J;X) \cap H^{k+1,1}(J;X) \colon u^{\prime} \in H^{k,q}_{\mu}(J;X)\},
\end{align*}
Here $u^{\prime}$ denotes  the time derivative of $u$ in the distributional sense.

The problem \eqref{eq:abstrpb} is considered for initial data within the real interpolation space
\begin{align*}
u_0\in X_{\gamma, \mu} =(X_0,X_1)_{\mu-1/q,q} \quad \hbox{and for} \quad  f\in \mathds{E}_{0,\mu}(J):= L_{\mu}^{q}(J;X_0), \quad 
\end{align*}
where  $q\in(1,\infty)$. We aim for solutions lying in the maximal regularity space 
\begin{align*}
\mathds{E}_{1,\mu}(J) := H^{1,q}_{\mu}(J;X_0) \cap L^q_{\mu}(J;X_1)
\end{align*}
and define for $ \beta \in [0,1]$ the space $X_\beta$ as the complex interpolation space $[X_0,X_1]_\beta$. 

The following existence and uniqueness results are based on the following assumptions:  
\begin{itemize}
	\item[(H1)] $A$ has maximal $L^q$-regularity for $q\in (1,\infty)$.
	\item[(H2)] $F\colon X_{\beta} \rightarrow X_0$ satisfied the estimate
	\begin{align*}
	\norm{F(u_1) - F(u_2)}_{X_0} \leq C (\norm{u_1}_{X_{\beta}}+ \norm{u_1}_{X_{\beta}}) (\norm{u_1-u_2}_{X_{\beta}})
	\end{align*}
	for some $C>0$ independent of $u_1,u_2$.
	\item[(H3)] $\beta - (\mu-1/q) \leq \frac{1}{2}(1-(\mu-1/q))$, that is $2\beta -1 + 1/q\leq \mu$.
	\item[(S)] $X_0$ is of class UMD, and the embedding 
	$$H^{1,q}(\R;X_0) \cap L^q(\R;X_1) \hookrightarrow H^{1-\beta,q}(\R;X_{\beta})$$
	is valid for each $\beta \in (0,1)$ and $q\in (1,\infty)$.
\end{itemize}

\begin{theorem}{\cite[Theorem 1.2]{PruessWilke2016}}. \label{thm:pruesswilke}
Assume that the assumptions $(H1), (H2), (H3)$ and $(S)$ hold and let 
	\begin{align*}
	u_0 \in X_{\gamma, \mu} \qquad \hbox{and} \qquad f \in L^q(0,T;X_0). 
	\end{align*}
	
	Then there exists a time $T'=T'(u_0)$ with $0<T'\leq T$ such that problem \eqref{eq:abstrpb} admits a unique solution 
	\begin{align*}
	u\in H^{1,q}_{\mu}(0,T';X_0) \cap L_{\mu}^q(0,T';X_1).
	\end{align*}
	Furthermore, the solution $u$ depends continuously on the data. 
\end{theorem}

\begin{remarks}\label{rem:pruesswilke}
	\begin{itemize}
\item[(a)] 
Condition $(S)$ holds true whenever $X_0$ is of class UMD and there is an operator $A_{\#} \in \mathcal{H}^{\infty}(X_0)$ with domain $D(A_{\#})=X_1$ satisfying  
$\phi^{\infty}_{A_{\#}} < \pi/2$, see Remark 1.1 of \cite{PruessWilke2016}.
\item[(b)] To verify condition $(S)$ in the situation considered here, i.e., $X_0=\lpso$ and $X_1=D(A_p)$, note first that $\lpso$ is of class UMD as closed subspace of $L^p(\Omega)^2$, 
and second, considering  $A_{\#}= A_p - \lambda$ with $\lambda>0$, it was proved in  \cite{GigaGriesHusseinHieberKashiwabara2016}
that $A_{\#} \in \mathcal{H}^{\infty}(X_0)$ with  $\phi^{\infty}_{A_{\#}} =0 < \pi/2$. 
\item[(c)] Due to the embeddings
		\begin{align*}
		\mathds{E}_{1,\mu}(0,T') \hookrightarrow C([0,T']; X_{\gamma,\mu}) \quad \hbox{and}\quad \mathds{E}_{1,\mu}(\delta,T') \hookrightarrow C([\delta,T']; X_{\gamma}), \quad \delta>0,
		\end{align*}
		there is an instantaneous smoothing effect typical for parabolic equations, compare e.g. \cite[Section 3.5.2]{PruessSimonett2016}.  
	\end{itemize}
\end{remarks}

When investigating the question of a global solution,  we consider
\begin{align*}
t_+(u_0):=\sup \{ T' >0 \colon \hbox{equation \eqref{eq:abstrpb} admits a solution on $(0,T')$}\}.
\end{align*}
By the above Theorem~\ref{thm:pruesswilke}, this set is non-empty, and we say that \eqref{eq:abstrpb} has a \textit{global solution} if for 
$f \in L^q(0,T;X_0)$ one has $t_+(u_0)=T$, where $0<T\leq \infty$. Global existence results can be derived from suitable a priori bounds 
following \cite[Theorem 5.7.1]{PruessSimonett2016}. The statement of \cite[Theorem 5.7.1]{PruessSimonett2016} and the relevant corollary \cite[Corollary 5.1.2]{PruessSimonett2016} 
are not formulated for the optimal time-weight used in \cite[Theorem 1.2]{PruessWilke2016}, however they carry over to this situation directly and without modifications. 
Therefore we state them here without proof. In the following, we denote by $C_b$ bounded continuous functions.

\begin{theorem}{\cite[Theorem 5.7.1]{PruessSimonett2016}}\label{thm:globex} 
Assume in addition to the assumptions of Theorem~\ref{thm:pruesswilke} %for $0<T^{\ast}\leq \infty$ and $f \in L^q(0,T^{\ast};X_0)$, assume 
that for $\mu < \overline{\mu} \leq 1$ the embedding  
	\begin{align*}
	X_{\gamma, \overline{\mu}} \hookrightarrow X_{\gamma, \mu}
	\end{align*}
	is compact, and that for some $\tau \in (0,t_+(u_0))$ the solution of \eqref{eq:abstrpb} satisfies
	\begin{align*}
	u\in C_b([\tau, t_+(u_0)); X_{\gamma, \overline{\mu}}),
	\end{align*}
	then there is a global solution to \eqref{eq:abstrpb}, i.e. $T'=T$.
\end{theorem}

Additional time and space-time regularity solutions of \eqref{eq:abstrpb} for more regular right hand side $f$ can be derived using the parameter trick based on the implicit 
function theorem.   Here for the time regularity, the version of of the parameter trick is adapted to time-weighted spaces, see e.g. \cite[Theorem 9.1]{ChillFasangova2009}, 
\cite[Section 9.4]{PruessSimonett2016} 

Let $F\colon X_1 \rightarrow X_0$ be a continuously differentiable function and $f$ an integrable function. We consider the problem
\begin{align}\label{eq:Fu}
u^{\prime} + F(u) = f.
\end{align}

\begin{theorem}{Compare \cite[Theorem 9.1]{ChillFasangova2009}}\label{thm:timereg}
	Let $q\in (1,\infty)$. Assume that for $k\in \N$, the  composition operator 
	$$\mathcal{F}\colon H_{\mu}^{1,q}(0,T;X_0) \cap L_{\mu}^{q}(0,T;X_1) \rightarrow L_{\mu}^{q}(0,T;X_0), \qquad u \mapsto F(u)$$
	is $k$ times continuously differentiable. Let $f\in L_{\mu}^{q}(0,T;X_0)$ and let $u\in H_{\mu}^{1,q}(0,T;X_0) \cap L_{\mu}^{q}(0,T;X_1)$ be a solution 
to \eqref{eq:Fu} on $(0,T)$. Assume that for the differential $D F$ of $F$ the linear problem 
	\begin{align*}
	u^{\prime} + D F(u)v = g, \qquad v(0)=0,
	\end{align*}
	admits for every $g \in L_{\mu}^{q}(0,T';X_0)$, $0<T'<T$ a unique solution $v\in H_{\mu}^{1,q}(0,T';X_0) \cap L_{\mu}^{q}(0,T';X_1)$. 
	
	Then for every $j = 0, \ldots, k$
	\begin{align*}
	u\in H_{loc}^{k+1,q}(0,T';X_0) \cap H_{loc}^{k,q}(0,T';X_1), \\
	t \mapsto t^j u^{(j)}(t) \in H_{\mu}^{1,q}(0,T';X_0) \cap L_{\mu}^{q}(0,T';X_1).
	\end{align*}
	If $\mathcal{F}$ and $f$ are of class $C^\infty$ or $C^{\omega}$, then $u \in C^{\infty}((0,T);X_1)$ or $u \in C^{\omega}((0,T);X_1)$, respectively.
\end{theorem}

\begin{remark}
In many versions of such regularity theorems the mapping property $\mathcal{F}\colon X_{\gamma,\mu} \rightarrow \mathcal{L}(X_1,X_0)$ is assumed, while the condition 
imposed in \cite[Theorem 9.1]{ChillFasangova2009} is weaker compared to other versions. 
\end{remark}

\section{Proofs of the main results}\label{sec:proofs}
In this section we prove our main results. For $1<p<\infty$ we set
\begin{align*}
X_0=\lpso \quad \hbox{and} \quad X_1=D(A_p).
\end{align*}

\subsection{Local well-posedness}
Similarly to \cite[Section 5]{HieberKashiwabara2015}, we define  for $1<p<\infty$ the bilinear map $F_p$ by 
\begin{align*}
F_p(v,v') := P_p(v \cdot \nabla_H v' + w(v)\partial_z v'),
\end{align*}
and set $F_p(v):=F_p(v,v)$. Since $u'=(v',w(v'))$ is divergence-free, we also obtain the representation 
$$
F_p(v,v') = P_p \div ( u'\otimes v).
$$
We start by collecting various facts concerning  the map $F_p$.

\begin{lemma}\label{lemma:Fp}
There exists a constant $C>0$, depending only on $\Omega$, $p\in (1,\infty)$ and $s\geq 0$, such that for $v,v'\in H^{s+1+1/p,p}(\Omega)^2$
	\begin{align*}
	\norm{F_p(v,v')}_{H^{s,p}(\Omega)^2}  &\leq C \norm{v}_{H^{s+1+1/p,p}}\norm{v'}_{H^{s+1+1/p,p}}, 
	\end{align*}
	i.e., $F_p(\cdot,\cdot)\colon H^{s+1+1/p,p}(\Omega)^2 \times H^{s+1+1/p,p}(\Omega)^2 \rightarrow H^{s,p}(\Omega)^2$ is a continuous bilinear map.	
\end{lemma}
\begin{proof}
The assertion is proved first inductively for the case $s=m\in \N_0$. A complex interpolation result for non-linear operators due to Bergh \cite{Bergh1984}, 
which, due to the bilinear structure of $F_p(\cdot,\cdot)$, is applicable completes the proof. 

The induction basis $m=0$ follows as in \cite[Lemma 5.1]{HieberKashiwabara2015} using anisotropic estimates and the bilinearity of $F_p(\cdot,\cdot)$. To prove the induction step, observe that 
\begin{align*}
\norm{F_p(v,v')}_{H^{m+1,p}} \leq \norm{\nabla F_p(v,v')}_{H^{m,p}}  + \norm{F_p(v,v')}_{H^{m,p}} 
\end{align*}
by using 
\begin{align*}
\partial_i F_p(v,v') = F_p(\partial_i v,v') + F_p(v,\partial_i v'), \quad i\in \{x,y,z\}  
\end{align*}
and $\norm{\partial_i v}_{H^{m+1+1/p,p}} \leq C \norm{v}_{H^{m+2+1/p,p}}$.
\end{proof}

Recalling that  $X_{\beta}$ is given by $X_\beta= [\lpso, D(A_p)]_{\beta}$ we obtain  
$$
[\lpso, D(A_p)]_{(1+1/p)/2}\subset H^{1+1/p,p}(\Omega)^2,
$$ 
which yields the following corollary of Lemma~\ref{lemma:Fp}.

\begin{corollary}\label{cor:nonlinearity}
Let  $\beta = \tfrac{1}{2}(1+1/p)$. Then there exists a constant $C>0$, independent of $v,v'$, such that
	\begin{align*}
	\norm{F_p(v)-F_p(v')}_{\lpso} \leq C \left(\norm{v}_{X_{\beta}} + \norm{v'}_{X_{\beta}} \right) \norm{v-v'}_{X_{\beta}}.
	\end{align*}
\end{corollary}

Local well-posedness for the primitive equations follows now from Theorem~\ref{thm:pruesswilke} by using Remarks~\ref{rem:pruesswilke} $(a)$ and $(b)$ for the conditions $(S)$ and $(H1)$ and 
Corollary~\ref{cor:nonlinearity} for the conditions $(H2)$ and $(H3)$. 

\begin{proposition}[Local well-posedness]\label{prop:loc}
		%Let $p,q\in (1,\infty)$ such that $1/p + 1/q\leq 1$. For $T>0$ let $\mu \in  [1/p + 1/q,1]$,
		% \begin{align*}
		% v_0\in X_{\mu-1/q,q} \qquad \hbox{and} \qquad P_p f \in H^{1,q}_{\mu}(0,T; \lpso) \cap  H^{1,2}(\delta,T; L^{2}_{\overline{\sigma}}(\Omega)) 
		%   \end{align*}	
%HIER GEHT ES WIEDER MIT DEN Ts DURECHEINANDER		
Let $p,q\in (1,\infty)$ with $1/p + 1/q\leq 1$, $\mu \in  [1/p + 1/q,1]$ and  $T>0$. Assume that
		\begin{align*}
		v_0\in X_{\mu-1/q,q} \qquad \hbox{and} \qquad P_p f \in L^{q}_{\mu}(0,T; \lpso).
		\end{align*}
		Then there exists $T'=T'(v_0)$ with $0<T'\leq T$ and a unique, strong solution $v$ to \eqref{eq:primequiv} on $(0,T')$ with
		\begin{align*}
		v\in H^{1,q}_{\mu}(0,T';\lpso) \cap L^{q}_{\mu}(0,T'; D(A_p)).
		\end{align*}
		
\end{proposition}

\subsection{Time and space regularity}

  %Verifying its assumptions consider
% \begin{align*}
 %\mathcal{F} \colon H_{\mu}^{1,q}(0,T;\lpso) \cap L_{\mu}^{q}(0,T;D(A_p)) \rightarrow L_{\mu}^{q}(0,T;\lpso), \quad v \mapsto A_p v + F_p(v),
 %\end{align*}
 %which is well-defined since first, by definition for $v \in L_{\mu}^{q}(0,T;D(A_p))$ one has
 % \begin{align*}
 % A_p v \in  L_{\mu}^{q}(0,T;\lpso),
 % \end{align*}
 %and second by Lemma \ref{lemma:nonlinearity} and using the interpolation inequality for $p\leq 2$ %and then Young's inequality and the embeddings %$X_{\mu,\gamma} \hookrightarrow \lpso$ and $\mathds{E}_{1,\mu}\hookrightarrow C([0,T];X_{\mu,\gamma})$
 %\begin{align*}
 %\int_0^T t^{1-\mu} \norm{F_p(v(t))}_{\lpso}^q dt &\leq \int_0^T t^{1-\mu} \norm{v(t)}^{2q}_{X_{\beta}} dt \leq %\int_0^T t^{1-\mu} \norm{v(t)}^{2q(1-1/p)}_{D(A_p)} \norm{v(t)}^{2q(1+1/p)}_{\lpso} dt  \\
 %&\leq (\sup_{0\leq t \leq T } \norm{v(t)}^{2q(1+1/p)}_{\lpso}) \int_0^T t^{1-\mu} \norm{v(t)}^{q}_{D(A_p)} dt \leq C %\norm{v}^{...}_{\mathds{E}_{1,\mu}} 
 %\end{align*}
 %and for $p>2$ using Young's inequality 
 %\begin{align*}
 %\int_0^T t^{1-\mu} \norm{F_p(v(t))}_{\lpso}^q dt &\leq \int_0^T t^{1-\mu} \norm{v(t)}^{2q}_{X_{\beta}} dt \leq %\int_0^T t^{1-\mu} \norm{v(t)}^{2q(1-1/p)}_{D(A_p)} \norm{v(t)}^{2q(1+1/p)}_{\lpso} dt  \\
 %& \leq \int_0^T t^{1-\mu} \norm{v(t)}^{q}_{D(A_p)}  dt + \int_0^T t^{1-\mu} \norm{v(t)}^{2q(1+1/p)}_{\lpso} dt \\
 %& \leq \norm{v}^q_{L_{\mu}^{q}(0,T;D(A_p))}   + C \norm{v}^{...}_{H_{\mu}^{1,q}(0,T;\lpso)} 
%\leq C \norm{v}^{...}_{\mathds{E}_{1,\mu}} 
% \end{align*}
Define $F(v):=A_p v + F_p(v)$. %Verifying the assumptions of Theorem~\ref{thm:timereg}, one proves

\begin{lemma}\label{lemma:F}
Let $p,q\in (1,\infty)$ with $1/p + 1/q\leq 1$, $\mu \in  [1/p + 1/q,1]$, $T>0$. Then the mapping 
 \begin{align*}
 \mathcal{F}\colon H_{\mu}^{1,q}(0,T;\lpso) \cap L_{\mu}^{q}(0,T;D(A_p)) \rightarrow L_{\mu}^{q}(0,T;\lpso), \quad v \mapsto F(v) %(F(v), v(0)) \times X_{\mu, \gamma}, 
 \end{align*} 
is continuously differentiable and even real analytic with
 \begin{align*}
%DF(v)\colon H_{\mu}^{1,q}(0,T;X_0) \cap L_{\mu}^{q}(0,T;D(A_p)) \rightarrow L_{\mu}^{q}(0,T;\lpso) \times X_{\mu, \gamma}, \quad v \mapsto  (A_p h + F_p(v,h) + F_p(h,v), h(0))
 DF(v)h = A_p h + F_p(v,h) + F_p(h,v).
 \end{align*}	
Moreover, for any $g\in  L_{\mu}^{q}(0,T;\lpso)$ the equation
\begin{align*}
 \partial_t h  - DF(v)h = g, \qquad h(0)=0,
\end{align*}
admits a unique solution $h\in H_{\mu}^{1,q}(0,T;\lpso) \cap L_{\mu}^{q}(0,T;D(A_p))$ with
\begin{align*}
\norm{h}_{H_{\mu}^{1,q}(0,T;\lpso) \cap L_{\mu}^{q}(0,T;D(A_p))} \leq C(v,g)
% \left(\norm{v}_{H_{\mu}^{1,q}(0,T;\lpso) \cap L_{\mu}^{q}(0,T;D(A_p))} + \norm{g}_{ L_{\mu}^{q}(0,T;\lpso)} \right),
\end{align*}	
where $C(v,g)>0$ remains bounded for $v\in H_{\mu}^{1,q}(0,T;\lpso) \cap L_{\mu}^{q}(0,T;D(A_p))$ and $g\in L_{\mu}^{q}(0,T;\lpso)$.
 \end{lemma}	
 
\begin{proof}
 In order to prove the existence of the Fr\'{e}chet derivative of $\mathcal{F}$, note that the bilinearity of $F_p(\cdot)$ yields  
 	\begin{align*}
 	\mathcal{F}(v+h) = \mathcal{F}(v) + A_p h+ F_p(v,h) + F_p(h,v) + F_p(h,h).
 	\end{align*}
Proceeding similarly to \cite{PruessWilke2016} we set for $h\in \mathds{E}_{1,\mu}(0,T)$
 	\begin{align*}
 	h_0^*:= e^{tA_p}h(0) \qquad \hbox{and} \qquad h':= h -h_0^*.
 	\end{align*}
Hence, $h= h_0^* + h'$ with $h_0^*(0)=h(0) \in X_{\gamma,\mu}$ and $h'(0)=0$. As observed in \cite{PruessWilke2016}, the assumption $(S)$, Sobolev's embeddings and Hardy's inequality, 
imply
 			\begin{align*}
 			\{u\in \mathds{E}_{1,\mu}(0,T)\mid u(0)=0 \} &\hookrightarrow  \{u\in H^{1-\beta,p}_{\mu}(0,T; X_{\beta}) \mid u(0)=0 \} \\ &\hookrightarrow \{u\in H^{1-\beta-1/2p,2p}_{\mu}(0,T; X_{\beta}) \mid u(0)=0 \}  \hookrightarrow L^{2p}_{\sigma}(0,T; X_{\beta}), 
 			\end{align*} 
 where we use that $1-\beta-1/2p +\mu = \tfrac{1}{2}(1+\mu) =\vartheta$. Note that the embedding constants are independent of $T$. Hence,
 	\begin{align*}
 	\norm{F_p(h,h)}_{\mathds{E}_{0,\mu}} &\leq \norm{F_p(h',h')}_{\mathds{E}_{0,\mu}} + \norm{F_p(h',h_0^*)}_{\mathds{E}_{0,\mu}} + \norm{F_p(h_0^*,h')}_{\mathds{E}_{0,\mu}}  + \norm{F_p(h_0^*,h_0^*)}_{\mathds{E}_{0,\mu}}  \\
 	&\leq C \left(\norm{h'}^2_{L^{2q}_{\vartheta}(0,T;X_{\beta})} + \norm{h'}_{L^{2q}_{\vartheta}(0,T;X_{\beta})}\norm{h_0^*}_{L^{2q}_{\vartheta}(0,T;X_{\beta})} +\norm{h^*_0}^2_{L^{2q}_{\vartheta}(0,T;X_{\beta})}\right) \\
 	&\leq C \left(\norm{h'}^2_{\mathds{E}_{1,\mu}} + \norm{h'}_{\mathds{E}_{1,\mu}}\norm{h(0)}_{X_{\gamma, \mu}} +\norm{h(0)}^2_{X_{\gamma,\mu}} \right) \leq C \norm{h}^2_{\mathds{E}_{1,\mu}},
 	\end{align*}
 	using $\norm{h^*_0}_{L^{2q}_{\vartheta}(0,T;X_{\beta})}\leq C \norm{h(0)}_{X_{\gamma,\mu}}\leq C \norm{h}_{\mathds{E}_{1,\mu}}$, and $\norm{h'}_{\mathds{E}_{1,\mu}}\leq \norm{h}_{\mathds{E}_{1,\mu}} + \norm{h^*_0}_{\mathds{E}_{1,\mu}}\leq C \norm{h}_{\mathds{E}_{1,\mu}}$, since $\norm{h^*_0}_{\mathds{E}_{1,\mu}} \leq C \norm{h(0)}_{X_{\gamma,\mu}} \leq C \norm{h}_{\mathds{E}_{1,\mu}}$, compare \cite[Proof of Proposition 3.4.2 and 3.4.3, Theorem 3.4.8]{PruessSimonett2016}.
For this reason,  $\norm{F(h,h)}_{\mathds{E}_{0,\mu}}/\norm{h}_{\mathds{E}_{1,\mu}} \to 0$ for $\norm{h}_{\mathds{E}_{1,\mu}}\to 0$. Similarly, we show that the map $v\mapsto DF(v)$ is continuous. 
Clearly, since $\mathcal{F}$ is quadratic in $v$, it is real analytic.    
 	
It remains to prove the (global) solvability in $\mathds{E}_{1,\mu}(0,T)$ of 
 	\begin{align*}
 	\partial_t h  - DF(v)h = g, \qquad h(0)=0.
 	\end{align*}
It seems that \cite[Theorem 1.2]{PruessWilke2016} cannot be applied directly because $v\mapsto DF(v)$ is well defined for $v\in \mathds{E}_{1,\mu}(0,T)$, but it is not 
for $v\in X_{\gamma,\mu}$. However, the assertion  can be proven by adapting the methods used in \cite{PruessWilke2016} and by using the linearity of the equation. 

First, we define the reference function $h_0^*\in \mathds{E}_{1,\mu}(0,T)$ for $h_0\in X_{\gamma,\mu}$ as the solution of the inhomogeneous linear problem
 	\begin{align*}
 		u^{\prime} + A_p u = g, \qquad u(0)=h_0, \quad \hbox{i.e.} \quad  h_0^*(t)= e^{tA_p}h_0 + \int_0^t e^{(t-s)A_p}g(s) ds.
 		\end{align*}
 Then, defining the ball
 \begin{align*}
 \mathds{B}_{r,T',h_0} := \{u\in \mathds{E}_{1,\mu}(0,T) \colon u(0)=h_0 \hbox{ and } \norm{u - h_0^*}_{\mathds{E}_{1,\mu}(0,T')} \leq r   \}  \subset\mathds{E}_{1,\mu}(0,T) ,  \quad r\in (0,1].
 \end{align*}
 we consider the map
 		\begin{align*}
 		\cT_{h_0} \colon  \mathds{B}_{r,T',h_0} \rightarrow \mathds{E}_{1,\mu}(0,T),  \quad \cT_{h_0} h = u,
 		\end{align*}
 where $u$ is the unique solution in $\mathds{E}_{1,\mu}(0,T)$ to the linear problem 
 		\begin{align*}
 		u^{\prime} + A_p u = g - F_p(v,h) - F_p(h,v), \qquad u(0)=h_0.
 		\end{align*}
 Choosing $r\in (0,1]$ and $0<T'\leq T$ appropriately, the mapping $\cT_{h_0}$ restricts to a contractive self-mapping on $\mathds{B}_{r,T',h_0}$.

As above, writing  $h= h' + h_0^*$ and $v= v' + v_0^*$, we obtain
\begin{align*}
\norm{F_p(v,h)+F_p(h,v)}_{\mathds{E}_{0,\mu}} \leq \norm{F_p(v,h)}_{\mathds{E}_{0,\mu}}  + \norm{F_p(h,v)}_{\mathds{E}_{0,\mu}} \leq C \norm{v}_{\mathds{E}_{1,\mu}} \left( \norm{h-h_0^*}_{\mathds{E}_{1,\mu}} + \norm{h(0)}_{X_{\gamma,\mu}} \right),
\end{align*} 	
and similarly
\begin{align*}
 	\norm{\cT_{h_0}h_1 -\cT_{h_0}h_2}_{\mathds{E}_{1,\mu}} &\leq C \norm{F(v,h_1-h_2) - F(h_1-h_2,v)}_{\mathds{E}_{0,\mu}} \\
 	&\leq C \norm{v}_{\mathds{E}_{1,\mu}} \left( \norm{h_1-h_2}_{\mathds{E}_{1,\mu}} \right).
 	\end{align*}
Now, $\norm{v}_{\mathds{E}_{1,\mu}(0,T)}$ is finite and $\norm{v}_{\mathds{E}_{1,\mu}(T_1,T_2)}\to 0$ for $|T_2-T_1| \to 0$. Therefore, for $\varepsilon>0$ there is a finite partition $0=T_0 < T_1 < \ldots <T_n=T$ of $(0,T)$ such that 
\begin{align*}
\norm{v}_{\mathds{E}_{1,\mu}(T_i,T_{i+1})} < \varepsilon, \quad  i \in \{0, \ldots n-1\}.
\end{align*}
Choosing $\varepsilon < 1/2C$ and by replacing -- using the linearity -- $h(T_i)$ and $g$ by $h_n(T_i)=h(T_i)/n$, $g_n=g/n$, respectively, $n\in \N$, and 
making relevant norms sufficiently small, we prove global existence iteratively in finitely many steps, where the norm of $h$ is controlled by $r$ in each step. 
Therefore, $\norm{h}_{\mathds{E}_{1,\mu}(0,T)}$ depends on the partition used above, in particular on the number of steps to reach the global solution and the norms of $h(T_i)$ and 
$g$ in $X_{\gamma,\mu}$ and $ L_{\mu}^{q}(0,T;\lpso)$, respectively.
 \end{proof}

\vspace{.2cm}\noindent 	
 \begin{proof}[Proof of Theorem~\ref{thm:time}]
The assertions $(a)$ and $(b)$ follow from Theorem~\ref{thm:timereg} by using Lemma~\ref{lemma:F}. 

Note that Theorem~\ref{thm:timereg} is based on the Banach space version of the implicit function theorem applied to maps of the type $(u(t,\cdot),\lambda)\mapsto \lambda F(u)(\lambda t, \cdot)$. Now, in order to prove $(c)$, the implicit function theorem needs to be applied to both space and time variables. For the most direct approach we need  that
 	\begin{align*}
 	\Phi_{\lambda, \eta} \colon  (0,\infty) \times \Omega\rightarrow (0,\infty) \times \Omega, \quad (t,x) \mapsto (\lambda \cdot t, x + t\eta), \quad \lambda \in (0,\infty), \quad \eta\in \R^3,
 	\end{align*} 
defines an isomorphism for parameters satisfying $\abs{\lambda - 1} <\epsilon$ and $\norm{\eta}<\epsilon$, $\epsilon>0$, see e.g.\cite[Section 5]{Pruess2002}. 
This is not true for general domains, but it is true for the whole space and also for the torus taking into account periodicity.  
 	
 	The strategy applied here is to first prove analyticity with respect to the horizontally periodic $x,y$ variables, thus proving analyticity of the pressure, and secondly, to 
apply a localization procedure in $z$ direction close to $\Gamma_D\neq \emptyset$, while on $\Gamma_N$ solutions are extended by even reflection onto a larger domain.     
 	
To this end, let $\Omega_{per} = \Omega\cup \Gamma_l$ be equipped with the topology of $S^1\times S^1 \times (-h,0)$, where $S^1=\R /\Z$, that is, taking into account 
lateral periodicity which induces a group structure in the lateral direction. Then  
 	\begin{align*}
 	\Phi_{\lambda, \eta_H} \quad \hbox{for} \quad \eta_H=(\mu_x,\mu_y,0), \quad \eta_x,\eta_y\in S^1
 	\end{align*}
   defines an isomorphism on $\Omega_{per}$  and for $v_{\lambda, \eta_H}= v\circ\Phi_{\lambda, \eta_H}$ we obtain by the chain rule
   	\begin{align*}
   	\partial_t v_{\lambda, \eta_H} = \lambda (\partial_t v)\circ \Phi_{\lambda, \eta_H} + \eta_H \cdot (\nabla v)\circ \Phi_{\lambda, \eta_H}.
   	\end{align*} 
Moreover, for $\epsilon >0$,  we define the real analytic map  
 		\begin{align*}
 		H\colon \mathds{E}_{1,\mu}\times (1-\epsilon, 1+\epsilon) \times (-\epsilon,\epsilon)^2 \rightarrow \mathds{E}_{0,\mu} \times X_{\gamma,\mu} % \quad (\lambda, \nu, u) \mapsto \lambda (\widetilde{A}_p v + \widetilde{F}_p(v))_{\lambda, \nu} + \mu \cdot (\nabla v)_{\lambda, \nu}
 		\end{align*}
 	by
 		\begin{align*}
 		H(\lambda, \eta_H, v):= (\partial_t v_{\lambda, \eta_H} -  \lambda (A_p v + F_p(v))_{\lambda, \eta_H}  - \eta_H \cdot \nabla v_{\lambda, \eta_H}, v_0 -v), %\qquad v_{\lambda, \eta_H}= v\circ\Phi_{\lambda, \eta_H},
 		\end{align*}
 		where the solution $v$ to \eqref{eq:primequiv} with initial data $v_0$ solves $H(1,0,v)=(0,0)$. 
 		
Note that $(A_p v + F_p(v))_{\lambda, \eta_H} = A_p v_{\lambda, \eta_H} + F_p(v_{\lambda, \eta_H})$. The Fr\'{e}chet derivative $\partial_v H$ is then an isomorphism by 
arguments similar to the ones given in the proof of Lemma~\ref{lemma:F} and by using that $H$ is polynomial in $v$. Therefore, the implicit function theorem yields 
that $v(\lambda t, x+ t \eta_H)$ is real analytic around $(1,0)$ in $\eta_H$ and $\lambda$. From this we  deduce real analyticity of $v$ around $(x,t)$ with respect to time and 
the horizontal directions, compare e.g. \cite[Section 5]{Pruess2002}. One can also adapt the approach in \cite{EscherSimonett2003} for 
locally symmetric spaces to the situation of a symmetry in only two space directions. In particular, this proves analyticity of the surface pressure $\pi_s$. 
 	
Concerning the $z$-direction, we note first that \eqref{eq:primequiv} is compatible with even reflections along the Neumann part of the boundary.
\begin{comment}
 	, assuming $\partial_zf\vert_{\Gamma_N}=0$, since if $E^{N}$ denotes the extension operator one has
 	\begin{align*}
 	& E^{N} (\partial_t v + v \cdot \nabla_H v + w(v) \cdot \partial_z v - \Delta v + \nabla_H \pi_s) \\
 	=& \partial_t E^{N}v + E^{N}v \cdot \nabla_H E^{N}v + w(E^{N}v) \cdot \partial_z E^{N}v - \Delta E^{N}v + \nabla_H \pi_s,
 	\end{align*} 
 	where one uses bi-linearity of the nonlinearity, the definition of $w(\cdot)$, and the compatibility of the pressure with even extension. 
\end{comment} 	
 Thus for $\Gamma_D=\emptyset$ solutions $v$ may be extended to the full torus -- a feature used in the literature dealing with Neumann boundary values, see e.g. \cite{LiTiti2016} -- 
and replacing $\eta_H$ by general $\eta\in \R^3$ in the above arguments implies  analyticity of solutions including the boundary. 
 	
 If $\Gamma_D\neq\emptyset$ we need to apply a localization procedure with respect to $z$-variable. The details of this method are neglected here and we refer to 
\cite[Section 9]{PruessSimonett2016} for details. 
 	
 Since the main non-locality in the primitive equation arise from the pressure term, we consider finally 
 	\begin{align*}
 	\partial_t v  - \Delta v  + v \cdot \nabla_H v + w(v) \cdot \partial_z v   = f_s , \quad f_s = f -\nabla_H \pi_s, \quad v(0)=v_0,
 	\end{align*} 
 	where $f_s$ is real analytic by the considerations above and the assumption on $f$.
The above proofs can now be adapted by using the fact that  the non-linearity $v \mapsto w(v)\partial_z + v\cdot \nabla_H v$ is real analytic.	
 \end{proof} 	
 
 \begin{remarks}
\begin{itemize}
\item[(a)]  A different strategy to prove smoothness, but not real analyticity, of solutions is to consider higher order time derivatives,  which are well defined according to 
Theorem~\ref{thm:time} (a) and (b). Then, by Lemma~\ref{lemma:Fp}
  		\begin{align*}
  		F_p(\cdot) \colon H^{s,p}(\Omega)^2 \rightarrow H^{s-(1+1/p),p}(\Omega)^2, \quad s \geq (1/2+1/2p), \quad p\in (1,\infty),
  		\end{align*}
  		is well-defined and bounded and  we write
  		\begin{align*}
  		\partial_t^{(n)} v  = A_p^{-1}( \partial_t^{(n)} F_p(v) - \partial_t^{(n+1)} v).
  		\end{align*}
 Since $\partial_t^{(n)} v\in D(A_p)$ for all $n\in \N_0$, we conclude first, that $\partial_t^{(n)} F_p(v)\in H^{2-(1+1/p),p}(\Omega)^2$. Secondly, applying elliptic regularity from 
Proposition~\ref{prop:ellipticregularity}, we conclude that $\partial_t^{(n)} v \in H^{(3 - 1/p),p}(\Omega)^2$. Iterating this argument, that is, 'trading time for space regularity' and 
using Sobolev embeddings we arrive at 
$$
v\in C^{\infty}((0,\infty); C_{per}^{\infty}(\overline{\Omega})^2),
$$
thereby proving smoothness including the boundary. \\
\item[(b)]  Another strategy for smoothness of solutions, namely proving first additional space regularity and deriving therefrom additional time regularity has been developed 
in \cite{GigaMiyakawa1985} in the case of the Navier-Stokes equations. 
\end{itemize}
\end{remarks}
 
The following elementary lemma is needed to extend regularity of solutions from $(0,T')$ for any $0<T'<T$ to $(0,T)$.

\begin{lemma}\label{lemma:T}
Let $v\in \mathds{E}_{1,\mu}(0,T')$ for any $0<T'<T$, and $\underset{0<T'<T}{\sup} \norm{v}_{\mathds{E}_{1,\mu}(0,T')}<C$ for some constant $C>0$. Then $v\in \mathds{E}_{1,\mu}(0,T)$. 
\end{lemma} 	

\begin{proof}
First, note that $v_t$ and $A_pv$ are measurable functions on $(0,T)$ by considering these as point-wise limits of the extensions by zero of $v_t\mid_{(0,T')}$ and 
$A_p v\mid_{(0,T')}$, respectively. Secondly, by dominated convergence 
\begin{align*}
\int_0^{T} t^{(1-\mu)q}\norm{v_t - A_p v}^q_{\lpso} = \lim_{T'\to T} \int_0^{T'} t^{(1-\mu)q}\norm{v_t - A_p v}^q_{\lpso} < \infty.
\end{align*}
\end{proof} 
	
\goodbreak 	
 	
 	\begin{lemma}\label{lemma:vt}
 	Let $p,q,\mu$ and $v_0, P_pf$ be as in Proposition~\ref{prop:loc}. Assume that
 		\begin{align*}
 		v\in H^{1,q}_{\mu}(0,T;\lpso) \cap L^{q}_{\mu}(0,T; D(A_p))
 		\end{align*}
    is a solution to \eqref{eq:primequiv}. If in addition $t \mapsto t\cdot P_p f_t \in L^{q}_{\mu}(0,T;\lpso)$, then 
 	\begin{align*}
 	t \cdot v_t \in H^{1,q}_{\mu}(0,T;\lpso) \cap L^{q}_{\mu}(0,T;D(A_p)), \quad \norm{t \cdot v_t}_{H^{1,q}_{\mu}(0,T;\lpso) \cap L^{q}_{\mu}(0,T;D(A_p))} \leq C(v, f, f_t, T),
 	\end{align*}
 	for some  finite constant $C$.
 	%In particular,
 	%\begin{align*}
 	%v \in H^{2,q}(\delta,T;\lpso) \cap H^{1,q}(\delta,T;D(A_p)), \quad v\in %C^0([\delta,T];D(A_p)),\quad \delta>0.
 	%\end{align*}
 	\end{lemma}

\begin{proof}
 	Let $v$ be a solution to \ref{eq:primequiv} in $\mathds{E}_{1,\mu}(0,T)$. Consider for $0<\varepsilon< \tfrac{T-T'}{T'}$, $0<T'<T$, %with $2T'>T$ 
 	the map  
\begin{align*}
G\colon (-\varepsilon, \varepsilon) \times \mathds{E}_{1,\mu}(0,T') \rightarrow \mathds{E}_{0,\mu}(0,T') \times X_{\gamma,\mu}, \quad (\lambda, \nu) \mapsto (\nu' +(1+\lambda)F(\nu) - (1+\lambda)f_{\lambda}, \nu(0)- v(0)),  
\end{align*} 	
 where $f_{\lambda}(t,\cdot):=f((1+\lambda)t,\cdot)$. As in \cite[Section 9.2]{ChillFasangova2009} one can prove that the implicit function theorem applies and there is an implicit function $g_{\lambda}(-\varepsilon',\varepsilon')\rightarrow \mathds{E}_{1,\mu}(0,T')$, $\varepsilon'\leq \varepsilon$ which solves $G(\lambda,g_{\lambda}(\lambda))=(0,0)$. By uniqueness we conclude that $g_{\lambda} = v_\lambda$. The implicit derivative at $\lambda=0$ is
 \begin{align*}
 	\partial_{\lambda}g_{\lambda}\vert_{\lambda=0}= t\cdot v_t = - (\partial_v G)(0,v) (A_p v +F_p(v) - f - t\cdot f_t,0),
 	\end{align*}
i.e., $t \mapsto -t\cdot v_t$ is the solution to the equation
 	\begin{align*}
 	\partial_t h - A_p h + F_p(h,v) + F_p(v,h) = A_p v +F_p(v) - f - t\cdot f_t, \quad h(0)=0, 
 	\end{align*}
 	and by Lemma~\ref{lemma:F}
 	\begin{align*}
 	\norm{t\cdot v_t}_{\mathds{E}_{1,\mu}(0,T')} \leq C(v, \norm{A_p v +F_p(v) - f - t\cdot f_t}_{\mathds{E}_{0,\mu}(0,T')}).
 	\end{align*}	
 Note that $\norm{A_p v +F_p(v) - f - t\cdot f_t}_{\mathds{E}_{0,\mu}} \leq C \left(\norm{v}_{\mathds{E}_{1,\mu}} + \norm{v}^2_{\mathds{E}_{1,\mu}} + \norm{f}_{\mathds{E}_{0,\mu}} + \norm{t\cdot f_t}_{\mathds{E}_{0,\mu}}\right)$. Now since $\norm{v}_{\mathds{E}_{1,\mu}(0,T)} + \norm{f}_{\mathds{E}_{0,\mu}(0,T)} + \norm{t\cdot f_t}_{\mathds{E}_{0,\mu}(0,T)}$ are bounded by assumption, $\sup_{T'<T}\norm{t\cdot v_t}_{\mathds{E}_{1,\mu}(0,T')}$ is bounded as well, and therefore $t\cdot v_t\in \mathds{E}_{1,\mu}(0,T)$ by Lemma~\ref{lemma:T}. %So, $v, v_t\in \mathds{E}_{1,\mu}(\delta,T)$, $\delta>0$. %,  and in particular $v\in H^{1,q}((\delta,T);D(A_p))$ and therefore $v\in C([\delta,T];D(A_p))$.  
 \end{proof}	
 
 	\begin{remark}
 		Note that in Lemma~\ref{lemma:vt} regularity of $t\cdot v_t$ is derived on $(0,T)$ while in Theorem~\ref{thm:time} it is proven on $(0,T')$ for any $T'<T$. Extending the regularity onto $(0,T)$ is possible due to the control on the implicit derivative by Lemma~\ref{lemma:F}. 
 	\end{remark}

\subsection{\textit{A priori} bounds in $H^1(0,T;L^2)\cap L^2(0,T;H^2)$}\label{sec:bdd}

\begin{theorem}[\textit{A priori bounds}]\label{thm: a priori estimate}
	There exists a continuous function 
	$B$ satisfying the following property:
	for any solution of \eqref{eq:primequiv} such that for $0<T <\infty$
		\begin{align*}
		v\in H^{1}(0,T;L^{2}_{\overline{\sigma}}(\Omega))) \cap L^{2}(0,T;D(A_2)), \quad v_0\in \{H^{1} \cap L^{2}_{\overline{\sigma}}(\Omega) \colon \restr{v}{\Gamma_D} = 0\}, \quad P_2 f \in L^2(0,T;L^{2}_{\overline{\sigma}}(\Omega))
		\end{align*}
     one has
   \begin{equation*}
    \norm{v}_{H^{1}(0,T;L^{2}_{\overline{\sigma}}(\Omega))) \cap L^{2}(0,T;D(A_2))} \leq B(\norm{v_0}_{H^1(\Omega)}, \norm{P_2 f}_{L^2(0,T;L^2(\Omega))}, T).
    %\max_{0\le t\le T} \|v(t)\|_{H^2(\Omega)} \le B(\|a\|_{H^2(\Omega)}, %\textcolor{blue}{\|f\|_{W^{1}(0,T;L^2(\Omega))}^2}).
   \end{equation*}
\end{theorem}

%\section{Global a priori $H^2$-bound}
%In this section we establish a global \textit{a priori} estimate for $v$ in the $H^2(\Omega)^2$-norm. %, which will be the key ingredient to prove the global well-posedness of \eqref{eq:primequiv}.

%now unavailable.

%\begin{theorem} %\label{thm: a priori estimate}
%	There exists a continuous function 
%	$B: \mathbb R\to \mathbb R$ satisfying the following property:
%	for any solution of \eqref{eq:primequiv} with $\Gamma_D=\emptyset$ such that $v \in C^1([0,T]; \lpso) \cap C([0,T]; D(A_2))$, with the initial data $v(0) = a$\textcolor{blue}{$\in D(A_2)\subset H^2(\Omega)$ and $f\in W^{1}((0,T);L^2(\Omega)^2)$}, one has
%	\begin{equation*}
%	\max_{0\le t\le T} \|v(t)\|_{H^2(\Omega)} \le B(\|a\|_{H^2(\Omega)}, \textcolor{blue}{\|f\|_{W^{1}(0,T;L^2(\Omega))}^2}).
%	\end{equation*}
%\end{theorem}
\begin{proof}
In \cite{GaldiHieberKashiwabara2015} global a priori bounds in $L^{\infty}(0,T;H^1(\Omega))$ and $L^{2}(0,T;H^2(\Omega))$ have been derived for the case of mixed Dirichlet and Neumann boundary condition. %provided that $v_0\in H^1$ with $v\mid_{\Gamma_D}=0$ and $P_2 f\in L^{2}(0,T;L^2(\Omega))$. 
	The case of pure Dirichlet boundary conditions can be treated similarly. Here, we supplement the corresponding proof for Neumann boundary conditions, where we even prove $L^{\infty}(0,T;H^2(\Omega))$-bounds.

	A standard procedure yields the energy equality: 
	\begin{equation} \label{eq: energy equality}
	\|v(t)\|_{L^2(\Omega)}^2 + 2\int_0^t \|\nabla v(s)\|^2 \, ds = \|a\|_{L^2(\Omega)}^2 + 2\int_0^t \int_{\Omega}f(s)\cdot v(s) \, ds.
	\end{equation}
We subdivide our proof into seven steps. The solution of \eqref{eq:primequiv} splits, compare \cite[(6.3) and (6.4)]{HieberKashiwabara2015}, into 
		\begin{align}
		\bar{v}_t-\Delta_H \bar{v}+ \nabla_H p &= - \bar{v}\cdot\nabla_H \bar{v} - \frac{1}{h}\int_{-h}^{0}   \tilde{v}\cdot\nabla_H \tilde{v} +  \mathrm{div}_H\tilde{v} \tilde{v} + \bar{f}, \qquad \mathrm{div}_H \bar{v}=0, \label{eq:vbar}\\
		% \mathrm{div}_H \bar{v} &=0 \\
		\tilde{v}_t-\Delta \tilde{v} + \tilde{v}\cdot\nabla_H \tilde{v} + w\partial_z \tilde{v}&= - \bar{v}\cdot\nabla_H \tilde{v} - \tilde{v}\cdot\nabla_H \bar{v} + \frac{1}{h}\int_{-h}^{0}   \tilde{v}\cdot\nabla_H \tilde{v} +  \mathrm{div}_H\tilde{v} \tilde{v} + \tilde{f}. \label{eq:vtilde}
		\end{align}	
The proof presented below basically follows the steps of \cite[Section 6]{HieberKashiwabara2015}. However, the Neumann boundary condition make the proof different in two ways.
One is that the extra term $\partial_zv|_{\Gamma_D}$ appearing in the equations for $\bar v$ and $\tilde v = v - \bar v$ is now absent (see (6.3) and (6.4) of \cite{HieberKashiwabara2015}), 
and the other is that Poincar\'e's inequality for $v$ is no longer available. However, we still have 
		\begin{align}\label{eq:poincare}
		\|\tilde v\|_{L^2(\Omega)} \le h\|\partial_z \tilde{v}\|_{L^2(\Omega)}. %\tilde{v}.
		\end{align}
		
	\textbf{Step 1.}
	We derive an estimate for $\tilde v: = v - \bar v \in L^\infty_t(L^4_x)$.
As in \cite[(6.8)]{HieberKashiwabara2015} and \cite[Step 3]{GaldiHieberKashiwabara2015}, by multiplying \eqref{eq:vtilde} with $|\tilde v|^2\tilde v$ and integrating over $\Omega$, we 
obtain by integrating by parts %with respect to the vertical direction}
	\begin{align*}
	&\frac14 \partial_t \|\tilde v\|_{L^4(\Omega)}^4 + \frac12 \big\|\nabla |\tilde v|^2 \big\|_{L^2(\Omega)}^2 + \big\| |\tilde v|\, |\nabla\tilde v| \big\|_{L^2(\Omega)}^2 \\
	=&\; -\int_\Omega (\tilde v\cdot\nabla_H\bar v) \cdot |\tilde v|^2\tilde v + \frac1h \int_\Omega \int_{-h}^0 (\tilde v\cdot\nabla_H\tilde v + \mathrm{div}_Hv \tilde v)\, dz \cdot |\tilde v|^2\tilde v + \int_{\Omega}\tilde{f} \cdot |\tilde v|^2\tilde v 
	=: I_1 + I_2.
	\end{align*}
We estimate
		\begin{align*}
I_1 \le&\; C\|\nabla_H\bar v\|_{L^2(G)} \|\tilde v\|_{L^4(\Omega)}^4 + C\|\nabla_H\bar v\|_{L^2(G)}^2 \|\tilde v\|_{L^4(\Omega)}^4 + \frac14 \big\|\nabla_H |\tilde v|^2 \big\|_{L^2(\Omega)}^2,
		\end{align*}
		and similarly to \cite[Section 4, Step 3]{GaldiHieberKashiwabara2015}
		\begin{align*}
		I_2 \le&\;  C (1 + \|\tilde{f} \|_{L^2(\Omega)}^2 )\|\tilde v\|_{L^4(\Omega)}^4  +\|\tilde{f} \|_{L^2(\Omega)}^{8/5}\|\tilde v\|_{L^4(\Omega)}^{12/5}   +  \frac14 \big\|\nabla_H |\tilde v|^2 \big\|_{L^2(\Omega)}^2
		\end{align*}
		as well as	
		\begin{align*}
		\|\tilde{f} \|_{L^2(\Omega)}^{8/5}\|\tilde v\|_{L^4(\Omega)}^{12/5} \leq C (1 + \|\tilde{v} \|_{L^4(\Omega)}^4 )\|\tilde{f} \|_{L^2(\Omega)}^2 + 1 + \|\tilde{v} \|_{L^4(\Omega)}^4.
		\end{align*}
Combining these estimates we obtain
		\begin{align*}
		\|\tilde v(t)\|_{L^4(\Omega)} %^{4/3} 
		+ \int_0^t \big\| |\tilde v(s)|\, |\nabla\tilde v(s)| \big\|_{L^2(\Omega)}^2 \, ds \le& 
		\|\tilde a \|_{L^4(\Omega)}  +	C \int_0^t(1+\|\tilde{f} \|_{L^2(\Omega)}^2)  \\
		+& C\left(1+ \|\nabla_H\bar v\|_{L^2(G)}^2 + \|\tilde{f} \|_{L^2(\Omega)}^2\right)  	\|\tilde v(t)\|_{L^4(\Omega)} ds,
		\end{align*}
		and Gronwall's inequality yields
		\begin{align*}
		\|\tilde v(t)\|_{L^4(\Omega)} \le& 
		\left(\|\tilde a \|_{L^4(\Omega)}  +	C (t+ \int_0^t\|\tilde{f} \|_{L^2(\Omega)}^2) \right) \exp \left(C \int_0^t \left(1+ \|\nabla_H\bar v\|_{L^2(G)}^2 + \|\tilde{f} \|_{L^2(\Omega)}^2\right) ds \right),
		\end{align*}	 	
		where $\int_0^t\|\nabla_H\bar v\|_{L^2(G)}^2 ds$ is bounded by \eqref{eq: energy equality}

	This implies that there exists a continuous function $B_1 = B_1(\|a\|_{H^1(\Omega)})$ such that
	%\textcolor{blue}{\footnote{\textcolor{blue}{Now since one has an $L_t^{\infty}L_x^4$ estimate for $\|\tilde v(t)\|_{L^4(\Omega)}$, one can consider the equation before dividing by $\|\tilde v\|_{L^4(\Omega)}^{8/3}$ and obtains also a bound for $\int_0^t\big\| |\tilde v|\, |\nabla\tilde v| \big\|_{L^2(\Omega)}^2$ }}}
	\begin{equation} \label{eq: L4 estimate for tilde v}
	\|\tilde v(t)\|_{L^4(\Omega)} %^{4/3} 
	+ \int_0^t \big\| |\tilde v(s)|\, |\nabla\tilde v(s)| \big\|_{L^2(\Omega)}^2 \, ds \le B_1(\|a\|_{H^1(\Omega)}, \|f\|_{L^2(0,T;L^2(\Omega))}), \qquad  t\in [0,T].
	\end{equation}
	
	\textbf{Step 2.}
	We derive an estimate for $\nabla_H\bar v \in L^\infty_t(L^2_x)$.
	As in \cite[p.\ 1103]{HieberKashiwabara2015} we obtain
	\begin{equation*}
	8\partial_t \|\nabla_H\bar v\|_{L^2(G)}^2 + \|\Delta_H\bar v\|_{L^2(G)}^2 + \|\nabla_H\pi\|_{L^2(G)}^2 \le C\big\| |\bar v|\,|\nabla_H\bar v| \big\|_{L^2(G)}^2 + C\big\| |\tilde v|\,|\nabla_H\tilde v| \big\|_{L^2(\Omega)}^2 +\|\bar{f}\|_{L^2(G)}^2.
	\end{equation*}
	By an interpolation inequality, $\|\bar v\|_{H^1(G)} = \|\bar v\|_{L^2(G)} + \|\nabla_H\bar v\|_{L^2(G)}$ etc., and $\|\nabla^2_H\bar v\|_{L^2(G)} \le C\|\Delta_H\bar v\|_{L^2(G)}$, we obtain
	\begin{align*}
	\big\| |\bar v|\,|\nabla_H\bar v| \big\|_{L^2(G)}^2 &\le C\|\bar v\|_{L^4(G)}^2 \|\nabla_H\bar v\|_{L^4(G)}^2
	\le C\|\bar v\|_{L^2(G)} \|\bar v\|_{H^1(G)} \|\nabla_H\bar v\|_{L^2(G)} \|\nabla_H\bar v\|_{H^1(G)} \\
	&\le C(1 + \|\bar v\|_{L^2(G)}^2 + \|\bar v\|_{L^2(G)}^4)\|\nabla_H\bar v\|_{L^2(G)}^2 + C\|\bar v\|_{L^2(G)}^2 \|\nabla_H\bar v\|_{L^2(G)}^4 + \frac12 \|\Delta_H\bar v\|_{L^2(G)}^2.
	\end{align*}
	It then follows from Gronwall's inequality and \eqref{eq: L4 estimate for tilde v} that
	\begin{equation*}
	\|\nabla_H\bar v(t)\|_{L^2(G)}^2 + \int_0^t \|\nabla_H\pi(s)\|_{L^2(G)}^2\, ds \le B_2(\|a\|_{H^1(\Omega)},\|f\|_{L^2(0,T;L^2(\Omega))}), \qquad  t\in[0,T]
	\end{equation*}
	for some continuous function $B_2$.
	
	\textbf{Step 3.} We derive an estimate for $v_z := \partial_z v \in L^\infty_t(L^2_x)$.
	As in \cite[(6.6)]{HieberKashiwabara2015} testing with $-\partial_{z}^2 v$ we obtain
	\begin{align*}
	\frac12 \partial_t\|v_z\|_{L^2(\Omega)}^2 &+ \|\nabla v_z\|_{L^2(\Omega)}^2 = -\int_\Omega v_z\cdot\nabla_H\bar v\cdot v_z + \int_\Omega \mathrm{div}_Hv_z\tilde v\cdot\tilde v_z \\
	&\hspace{3cm} + \int_\Omega v_z\cdot\nabla_Hv_z\cdot \tilde v - 2\int_\Omega \tilde v\cdot\nabla_Hv_z\cdot v_z - \int_\Omega f \cdot v_{zz} \\
	&\le C\|\nabla_H\bar v\|_{L^2(G)}\|v_z\|_{L^4(\Omega)}^2 + C\|\nabla_Hv_z\|_{L^2(\Omega)} \big\| |\tilde v|\, |v_z| \big\|_{L^2(\Omega)} +\frac14\|v_{zz}\|_{L^2(\Omega)}^2 + \|f\|_{L^2(\Omega)}^2  \\
	&\le C\|\nabla_H\bar v\|_{L^2(G)}^4 \|v_z\|_{L^2(\Omega)}^2 + C\big\| |\tilde v|\, |\nabla\tilde v| \big\|_{L^2(\Omega)}^2 + \frac12\|\nabla v_z\|_{L^2(\Omega)}^2 + \|f\|_{L^2(\Omega)}^2,
	\end{align*}
	where we have used $\|v_z\|_{L^4(\Omega)} \le C\|v_z\|_{L^2(\Omega)}^{1/4} \|\nabla v_z\|_{L^2(\Omega)}^{3/4}$ (note that $v_z = 0$ on $\Gamma_u\cup\Gamma_b$) and $v_z = \tilde v_z$.
	It follows from \eqref{eq: L4 estimate for tilde v} that
	\begin{equation*}
	\|v_z(t)\|_{L^2(\Omega)}^2 + \int_0^t\|\nabla v_z(s)\|_{L^2(\Omega)}^2\, ds \le B_3(\|a\|_{H^1(\Omega)}, \|f\|_{L^2(0,T;L^2(\Omega))}), \qquad  t\in[0,T].
	\end{equation*}
	
	\textbf{Step 4.} We derive an estimate for $\nabla v \in L^\infty_t(L^2_x)$.
	As in \cite[(6.13)]{HieberKashiwabara2015} we obtain
	\begin{align}
	\partial_t\|\nabla v\|_{L^2(\Omega)}^2 + \|\Delta v\|_{L^2(\Omega)}^2 &\le C( \|\bar v\cdot\nabla_H\bar v\|_{L^2(G)}^2 + \|\bar v\cdot\nabla_H\tilde v\|_{L^2(\Omega)}^2 + \|\tilde v\cdot\nabla_H\bar v\|_{L^2(\Omega)}^2 + \|w\partial_zv_z\|_{L^2(\Omega)}^2 \notag \\
	&\hspace{1cm} + \|\tilde v\cdot\nabla_H\tilde v\|_{L^2(\Omega)}^2 + \|\nabla_H\pi\|_{L^2(G)}^2 +\|f\|_{L^2(\Omega)}^2). \label{eq1: L2 estimate for Dv}
	\end{align}
	In view of interpolation inequalities, elliptic regularity for $\Delta$, and anisotropic estimates, we may bound the first four terms on the right-hand side as
	{\allowdisplaybreaks\begin{align*}
		&\begin{aligned}
		\bullet\; \|\bar v\cdot\nabla_H\bar v\|_{L^2(G)}^2 &\le C\|\bar v\|_{L^2(G)} \|\bar v\|_{H^1(G)} \|\nabla_H\bar v\|_{L^2(G)} \|\nabla_H\bar v\|_{H^1(G)} \\
		&\le C\|\bar v\|_{L^2(G)} \|\bar v\|_{H^1(G)} (1 + \|\bar v\|_{L^2(G)} \|\bar v\|_{H^1(G)}) \|\nabla_H\bar v\|_{L^2(G)}^2 + \frac18 \|\Delta v\|_{L^2(\Omega)}^2,
		\end{aligned}
		\\
		&\begin{aligned}
		\bullet\; \|\bar v\cdot\nabla_H\tilde v\|_{L^2(\Omega)}^2 &= \|\bar v\cdot \nabla_H \tilde{v}\|_{L^2(\Omega)}^2 \le C\|\bar v\|_{L^6(G)}^2 \|\nabla_H\tilde{v}\|_{L^2(\Omega)} \|\nabla_H\tilde{v}\|_{H^1(\Omega)} \\
		&\le C\|\bar v\|_{H^1(G)}^2 \|\nabla v\|_{L^2(\Omega)} \|\partial_z\nabla_Hv\|_{L^2(\Omega)} \\
		&\le C\|\bar v\|_{H^1(G)}^2 \|\nabla v\|_{L^2(\Omega)}^2 + \frac18 \|\Delta v\|_{L^2(\Omega)}^2,
		\end{aligned}
		\\
		&\begin{aligned}
		\bullet\; \|\tilde v\cdot\nabla_H\bar v\|_{L^2(\Omega)}^2 \le C\|\tilde v\|_{L^4(\Omega)}^2 \|\nabla_H\bar v\|_{L^4(G)}^2 \le C\|\tilde v\|_{L^4(\Omega)}^2(1 + \|\tilde v\|_{L^4(\Omega)}^2)\|\nabla_H\bar v\|_{L^2(G)}^2 + \frac18 \|\Delta v\|_{L^2(\Omega)}^2,
		\end{aligned}
		\\
		&\begin{aligned}
		\bullet\; \|wv_z\|_{L^2(\Omega)}^2 &\le C\|w\|_{L^4_{xy}L^\infty_z} \|v_z\|_{L^4_{xy}L^2_z} \le C\|\mathrm{div}_H\,v\|_{L^2(\Omega)} \|\mathrm{div}_Hv\|_{H^1(\Omega)} \|v_z\|_{L^2(\Omega)} \|\nabla v_z\|_{L^2(\Omega)} \\
		&\le C\|\mathrm{div}_Hv\|_{L^2(\Omega)}^4\|v_z\|_{L^2(\Omega)}^2 + C\|\mathrm{div}_Hv\|_{L^2(\Omega)}^2\|v_z\|_{L^2(\Omega)}^2\|\nabla v_z\|_{L^2(\Omega)}^2 + \frac18 \|\Delta v\|_{L^2(\Omega)}^2.
		\end{aligned}
		\end{align*}
	} 	%\textcolor{blue}{\footnote{\textcolor{blue}{I replaced $\widetilde{\nabla_Hv}$ by $\nabla_H\tilde{v}$ since $\frac{1}{h}\int_{-h}^{0} \nabla_H v = \frac{1}{h} \nabla_H \int_{-h}^{0} v = \nabla_H \bar{v}$, and hence also $\widetilde{\nabla_H v} = \nabla_H v - \frac{1}{h}\int_{-h}^{0} \nabla_H v = \nabla_H v - \nabla_H \bar{v} = \nabla_H (v-\bar{v})= \nabla_H \tilde{v}$ }}}
	
	Combining these with \eqref{eq1: L2 estimate for Dv} and applying Gronwall's inequality, we conclude
	\begin{equation} \label{eq2: L2 estimate for Dv}
	\|\nabla v(t)\|_{L^2(\Omega)}^2 + \int_0^t \|\Delta v(s)\|_{L^2(\Omega)}^2\, ds \le B_4(\|a\|_{H^1(\Omega)}, \|f\|_{L^2(0,T;L^2(\Omega))}), \qquad  t\in[0, T].
	\end{equation}
%	Since $\partial_t v = -v\cdot\nabla_Hv - w\partial_zv + \Delta v + \nabla_H\pi$ it also follows that
%	\begin{equation*}
%	\int_0^t \|\partial_t v(s)\|_{L^2(\Omega)}^2\, ds \le B_5(\|a\|_{H^1(\Omega)},\textcolor{blue}{\|f\|_{L^2(0,T;L^2(\Omega))}}), \qquad  t\in[0, T].
%	\end{equation*}

	 Having now established $L^{\infty}(0,T;H^1(\Omega))$ and $L^{2}(0,T;H^2(\Omega))$-{a priori} bounds $B_4$ for the all boundary conditions \ref{eq:bc}, we conclude, by using maximal 
regularity of $A_2$ that
	 \begin{align*}
	 \norm{v}_{\mathds{E}_{1,1}(0,T)} \leq c \norm{(\partial_t - (A_2-\lambda)v}_{\mathds{E}_{0,1}(0,T)} \leq C \norm{F_2(v) - \lambda v + f}_{\mathds{E}_{0,1}(0,T)}, \quad \lambda>0,
	 \end{align*} 
	 and by using Lemma \ref{lemma:Fp}, interpolation inequality and H\"older's inequality 
	 \begin{align*}
\int_{0}^T \norm{F_p(v(s))}^2_{L^2(\Omega)} ds \leq C \int_{0}^T \norm{v(s)}_{H^1}^2\norm{v(s)}_{H^2}^2 ds \leq C \norm{v}^2_{L^{\infty}(0,T;H^1(\Omega))}\norm{v}^2_{L^{2}(0,T;H^2(\Omega))}\leq B_4^4=:B,
	 \end{align*}
	 an {a priori} bound in the maximal regularity space. 
	 
	 To derive in addition an $L^{\infty}(0,T;H^2(\Omega))$-bounds we proceed as follows.

	\textbf{Step 5.} We derive an estimate for $v_t := \partial_t v \in L^\infty_t(L^2_x)$.
	Taking the time derivative of \eqref{eq:primequiv}, multiplying by $v_t$, and integrating over $\Omega$, we obtain using the divergence free condition
	\begin{equation} \label{eq1: L2 estimate for vt}
	\frac12\partial_t\|v_t\|_{L^2(\Omega)}^2 + \|\nabla v_t\|_{L^2(\Omega)}^2 = - \int_\Omega (v_t\cdot \nabla_Hv + w_t\partial_zv) \cdot v_t + \int_\Omega f_t \cdot v_{t}.
	\end{equation}
%	\textcolor{blue}{\footnote{\textcolor{blue}{Notation $\partial_t w\partial_zv$ can be misleading since it can use derivative of the product as well. I suggest $w_t\partial_zv$}}}
	In view of interpolation inequalities and anisotropic estimates, the first two terms on the right-hand side may be bounded by
	\begin{equation*}
	\|\nabla_Hv\|_{L^2(\Omega)} \|v_t\|_{L^4(\Omega)}^2 \le C(\|\nabla_Hv\|_{L^2(\Omega)} + C\|\nabla_Hv\|_{L^2(\Omega)}^4) \|v_t\|_{L^2(\Omega)}^2 + \frac14\|\nabla v_t\|_{L^2(\Omega)}^2,
	\end{equation*}
	and by
	\begin{align*}
	&\|w_t\|_{L^2_{xy}L^\infty_z} \|v_z\|_{L^3_{xy}L^2_z} \|v_t\|_{L^6_{xy}L^2_z} \\
	\le\; &C\|\mathrm{div}_Hv_t\|_{L^2(\Omega)}\|v_z\|_{H^{1/3}(\Omega)} \|v_t\|_{H^{2/3}(\Omega)} \\
	\le\; &C\|\nabla v_t\|_{L^2(\Omega)} \|v_z\|_{L^2(\Omega)}^{2/3} \|\nabla v_z\|_{L^2(\Omega)}^{1/3} \|v_t\|_{L^2(\Omega)}^{1/3} \|v_t\|_{H^1(\Omega)}^{2/3} \\
	\le\; &C\|\nabla v_z\|_{L^2(\Omega)}^2\|v_t\|_{L^2(\Omega)}^2 + C\|v_z\|_{L^2(\Omega)}^{4} \|\nabla v_z\|_{L^2(\Omega)}^{2} \|v_t\|_{L^2(\Omega)}^{2} + \frac14\|\nabla v_t\|_{L^2(\Omega)}^2,
	\end{align*}
	respectively.  Therefore, integrating \eqref{eq1: L2 estimate for vt} with respect to $t$ and noting that $\|v_t(0)\|_{L^2(\Omega)} \le C(\|a\|_{H^2(\Omega)}^2 + \|a\|_{H^2(\Omega)})$ (see \cite[p.\ 1111]{HieberKashiwabara2015}), we conclude $\|v_t(t)\|_{L^2(\Omega)} \le B_6(\|a\|_{H^2(\Omega)}, \|f\|_{W^{1}(0,T;L^2(\Omega))}^2)$ for all $t\in[0, T]$.
	
	\textbf{Step 6.} We derive an estimate for $v_z \in L^\infty_t(L^3_x)$.
	As in \cite[p.\ 1109]{HieberKashiwabara2015} we obtain testing with $-\partial_z (|v_z|v_z)$, now assuming $f_z \in L^{2}(0,T;L^2(\Omega))$
	\begin{align*}
	&\frac13\partial_t\|v_z\|_{L^3(\Omega)}^3 + \frac49 \big\| \nabla|v_z|^{3/2} \big\|_{L^2(\Omega)}^2 + \big\| |v_z|^{1/2}|\nabla v_z| \big\|_{L^2(\Omega)}^2 \\
	=\; &-\int_\Omega v_z\cdot\nabla_Hv\cdot |v_z|v_z + \int_\Omega \mathrm{div}_Hv |v_z|^3 - \int_{\Omega} f\cdot\partial_z (|v_z|v_z) \\
	\le &C\|\nabla_Hv\|_{L^2(\Omega)}^4\|v_z\|_{L^3(\Omega)}^3 + \frac19 \big\| \nabla|v_z|^{3/2} \big\|_{L^2(\Omega)}^2 + \|f_z\|_{L^2(\Omega)}^2 + \|v_z\|_{L^4(\Omega)}^4,
	\end{align*}
where $\|v_z\|_{L^4(\Omega)}^4 \leq C (\|v_z\|_{L^3(\Omega)}^3 +1)(\|\nabla v_z\|_{L^2(\Omega)}^2 +1)$. Gronwall's inequality then implies $\|v_z(t)\|_{L^3(\Omega)} \le B_7(\|a\|_{H^2(\Omega)},  \|f_z\|_{L^{2}(0,T;L^2(\Omega))}^2)$ for all $t\in[0, T]$.
	
	\textbf{Step 7.} We now derive an estimate for $\nabla^2v \in L^\infty_t(L^2_x)$.
	As in \cite[p.\ 1111]{HieberKashiwabara2015} we have
	\begin{align*}
	\|\nabla^2v\|_{L^2(\Omega)} &\le C\|\Delta v\|_{L^2(\Omega)} \le C(\|v_t\|_{L^2(\Omega)} + \|v\cdot\nabla_Hv\|_{L^2(\Omega)} + \|w\partial_zv\|_{L^2(\Omega)}) + \|f\|_{L^2(\Omega)}^2 \\
	&\le C\|v_t\|_{L^2(\Omega)} + C\|v\|_{L^6(\Omega)}\|v\|_{W^{1,3}(\Omega)} + C\|w\|_{L^6(\Omega)}\|v_z\|_{L^3(\Omega)} + \|f\|_{L^2(\Omega)}^2 \\
	&\le C\|v_t\|_{L^2(\Omega)} + C\|v\|_{H^1(\Omega)}^3 + C\|\nabla v\|_{L^2(\Omega)} (\|v_z\|_{L^3(\Omega)} + \|v_z\|_{L^3(\Omega)}^3)  + \|f\|_{L^2(\Omega)}^2 + \frac12 \|\nabla^2v\|_{L^2(\Omega)},
	\end{align*}
	which implies the desired estimate
	\begin{equation} \label{eq: L2 estimate for DDv}
	\|\nabla^2v(t)\|_{L^2(\Omega)} \le B_8(\|a\|_{H^2(\Omega)}, \|f\|_{W^{1}(0,T;L^2(\Omega))}^2, \|f_z\|_{L^{2}(0,T;L^2(\Omega))}^2) \qquad \forall t\in[0, T].
	\end{equation}
	Combining \eqref{eq: energy equality}, \eqref{eq2: L2 estimate for Dv}, and \eqref{eq: L2 estimate for DDv}, we completed the proof.% of Theorem \ref{thm: a priori estimate}.

\end{proof}

\subsection{Global well-posedness}

\begin{proof}[Proof of Proposition~\ref{thm:globl2}]
To prove assertion  (a) consider 
\begin{align*}
t_+(v_0):=\sup \{ T'>0 \colon \hbox{Equation \eqref{eq:primequiv} has a solution in  $\mathds{E}_{1,1}(0,T')$}\}.
\end{align*}
By Proposition~\ref{prop:loc} $t_+(v_0)>0$ and the solutions in $\mathds{E}_{1,1}(0,T')$ are unique. Indeed, if we assume that there are two solutions $v,v'\in \mathds{E}_{1,1}(0,T')$, 
then setting  
$$t_{1}(v_0):=\sup \{ s>0 \colon \norm{(v-v')(s)}_{X_{\gamma,1}} = 0\},$$ 
we see that  $t_{1}(v_0)>0$ by Proposition~\ref{prop:loc}. Further, by continuity, $\mathds{E}_{1,1}(0,T')\hookrightarrow C([0,T'];X_{\gamma,1})$ and the above supremum is attained. 
Assuming that $t_{1}(v_0)<T'$, again by Proposition~\ref{prop:loc}, the solution with new initial value at $t_{1}(v_0)$ is unique on some time interval, thus contradicting the assumption. 

Assume now, that $t_+(v_0) < T$. By Theorem~\ref{thm: a priori estimate} $\norm{v}_{\mathds{E}_{1,1}(0,T')}\leq B(\norm{v_0}_{H^1(\Omega)}, \norm{P_2 f}_{L^2(0,T;L^2(\Omega))}, t_+(v_0))$ for any $0<T'<t_+(v_0)$. Hence by Lemma~\ref{lemma:T} we have $v\in \mathds{E}_{1,1}(0,t_+(v_0))$. Since the trace in $\mathds{E}_{1,1}(0,t_+(v_0))$ is well-defined $v(t_+(v_0))$ can be taken as new initial value, thus extending the solution beyond $t_+(v_0)$ contradicting the assumption. Hence $t_+(v_0) = T$, and again combing Theorem~\ref{thm: a priori estimate} and Lemma~\ref{lemma:T} we have  $v\in \mathds{E}_{1,1}(0,T)$. This proves part (a).    

Assertion  (b) follows directly from Lemma~\ref{lemma:vt}.
\end{proof}

\begin{proof}[Proof of Theorem~\ref{thm:glob}]
By Proposition~\ref{prop:loc} there is a local solutions, which by Theorem~\ref{thm:time} (a) has additional time regularity, in particular $v \in H^{1,q}(\delta,T; D(A_p))\hookrightarrow C^0(\delta,T; D(A_p))$ for some $0<\delta\leq T'$ and $0<T'< T$. Now, using $v(T')$ as new initial value, and taking advantage of the embedding $D(A_q) \subset (L_{\overline{\sigma}}^2(\Omega),D(A_2))_{1/2,q}$ for $q\in [6/5,\infty)$ and the additional assumption $P_2 f \in W^{1,2}(\delta,T; L_{\overline{\sigma}}^2(\Omega))$ we obtain that $v$ is also an $L^2$ solution at least for $\delta>0$. This holds for $q\in [6/5, \infty)$, and for $q\in (1,6/5)$ this follows from a bootstrapping argument as in \cite[Section 6.2]{HieberKashiwabaraHussein2016}. By Proposition~\ref{thm:globl2} there exists a global $L^2$ solution with $v\in C_b(\delta,D(A_2))$. Then using Lemma \ref{lem:real} and classical embedding results, see e.g. \cite{Triebel}, we obtain
\begin{align*}
D(A_2) \hookrightarrow X_{\overline{\mu},q} \quad \hbox{for } 0 \leq \overline{\mu}-\mu < 2 - \tfrac{2}{p},
\end{align*}
and compactness of the embedding $X_{\overline{\mu},q} \hookrightarrow X_{\mu,q}$ for $1/p <\mu <\overline{\mu}<1$. %follows from Lemma \ref{lem:real} and classical embedding results. % using that  
%\begin{align*}
%B^s_{pq}(\Omega) \hookrightarrow B^{s-\epsilon_1}_{p2}(\Omega) \hookrightarrow H^{s-\epsilon_1,p}(\Omega) 
% \hookrightarrow H^{s-\epsilon_2,p}(\Omega) ...  
%\end{align*}
%or cite...
Hence Theorem~\ref{thm:globex} applies since 
\begin{align*}
\norm{v}_{C_b (\delta,T;X_{\overline{\mu},q})} \leq C \norm{v}_{C_b (\delta,T; D(A_2))}, %\leq C B(T),
\end{align*}
and therefore the solution exists globally, that is for any $T>0$.
\end{proof}

%Having an $H^1(L^2)\cap L^2(H^2)$ \textit{a priori} bound on $v$ it follows from Theorem~\ref{thm:vt} that $\norm{A_2 v(t)}_{L^2} \leq C(v, f,f_t,T)$, $0<t\leq T$, where $C(v, f,f_t,T)>0$ is bounded if $v\in \mathds{E}_{1,\mu}(0,T)$.

%In particular $v \in W^{1,q}(\delta,T; D(A_p))\hookrightarrow C^0(\delta,T; D(A_p))$.  Now, for $v_0\in D(A_2), P_2 f \in W^{1,2}(\delta,T; L_{\overline{\sigma}}^2(\Omega))$ there is an $H^2$ \textit{a priori} bound 
%\begin{align*}
%\norm{v}_{C_b (\delta,T; D(A_2))} \leq B(T),
%\end{align*}
%where $B(\cdot)$ is a continuous function depending only on $T$, $\norm{v(\delta)}_{D(A^2)}$ and $%\norm{P_2 f}_{W^{1,2}(\delta,T; L_{\overline{\sigma}}^2(\Omega))}$, see \cite[Section 6]{HieberKashiwabara2015} for the case of mixed Dirichlet and Neumann boundary conditions, the proof for pure Dirichlet boundary conditions analogous while the proof for pure Neumann boundary conditions needs several adjustments, and it is given in the subsequent Section \ref{sec:bdd}.%
%Then 

%\section*{Overview: Initial conditions and regularity}

%\section*{Appendix: Proofs for semilinear evolution equations and maximal $L^q$-regularity }

\section{Concluding Remarks}\label{sec:concludingremark}

The maximal regularity approach uses the contraction mapping principle to construct local solutions with initial values being traces of functions in the maximal $L^q$-$L^p$-regularity class which here reads as $v_0\in B^{2/p}_{pq}$. Other methods to construct solutions for the primitive equations are the Fujita-Kato scheme as proposed in \cite{HieberKashiwabara2015} for initial values $v\in H^{2/p,p}$, and the Galerkin method as used originally in \cite{Guillenetall2001} giving initial values $v_0\in H^1$. Note that for $q=p=2$ all results agree, and for $p,q\geq 2$ one has $H^{2/p,p}\subset B^{2/p}_{pq}$.

%Comparing the methods, one notices first that the Galerkin method is restricted to Hilbert spaces, while Fujita-Kato as well as maximal $L^q$ regularity are applicable in Banach spaces, too. Concerning the global \textit{a priori} bounds, for the Galerkin scheme it is sufficient to prove an \textit{a priori} bound in the space of initial values, see \cite{CaoTiti2007}, i.e. $H^1$, since then the approximate Galerkin solutions converge in this space. For both the Fujita-Kato and maximal $L^q$ regularity bounds are needed with slightly more regularity than the admissible initial values, compare \cite[Remark 5.4]{HieberKashiwabara2015} and Remark~\ref{rem:main} above. Since \textit{a priori} bounds are proven usually in $L^2$ spaces one overcomes this difficulty by using $H^2$ \textit{a priori} bounds. 

Comparing the maximal regularity and the Fujita-Kato approach, we see that, by using the maximal regularity approach, various boundary conditions can be treated simultaneously in the same way. 
The efficiency of this approach becomes furthermore obvious, when studying further couplings adding to the complexity of the equations. For instance, adding non-constant 
temperature $\tau$ one considers
\begin{align*}
\left\{
\begin{array}{rll}
\partial_t v + v \cdot \nabla_H v + w \cdot \partial_z v - \Delta v + \nabla_H \pi_s+  \nabla_H \int_{-h}^z\tau(\cdot,\xi) d\xi  & = f  , &\text{ in } \Omega \times (0,T),  \\
\mathrm{div}_H \overline{v} & = 0, &\text{ in } \Omega \times (0,T),   \\
\partial_t \tau + v \cdot \nabla_H \tau + w \cdot \partial_z \tau - \Delta \tau & = g, &\text{ in } \Omega \times (0,T),
\end{array}\right.
\end{align*}
compare \cite{Lionsetall1992}, where the non-linearity 
$$F_p(v,\tau):=\left(P_p( v \cdot \nabla_H v + w \cdot \partial_z v +  \nabla_H \int_{-h}^z\tau(\cdot,\xi) d\xi), v \cdot \nabla_H \tau + w \cdot \partial_z \tau \right)$$
can be estimated as in \cite[Lemma 5.1]{HieberKashiwabaraHussein2016}, and local well-posedness and regularity results follow directly.

Recently, the coupling to moisture and its analysis has come into focus, see \cite{Cotietall2015, Hitmeieretall2016} and the references given therein. The equation for the moisture $q$ is of the type   
    \begin{align*}
    \partial_t q + v \cdot \nabla_H q + w \cdot \partial_z q - \Delta q  = h + F(v, \tau, q)
    \end{align*}   
with additional coupling term $F(v, \tau, q)$. In the model studied in \cite{Cotietall2015} $F(v, \tau, q)$ involves some Heaviside functions and it is treated using variational methods. In  \cite{Hitmeieretall2016} water vapor $q_v$, cloud water $q_c$ and rain water $q_r$ mixing ratios are coupled to the temperature and velocity equations where the coupling terms 
involve expressions of the form
$$\tau(q_r^+)^{\beta}(q_{vs} - q_v), \quad \beta \in (0,1], \quad q_r^+=\max\{0,q_r\}$$ 
%\section{Fujita-Kato and maximal $L^q$-regularity compared}\label{sec:FujitaKato} 
for fixed saturation mixing ratio $q_{vs}$. 
For $\beta=1$ this is Lipschitz continuous, and hence maximal $L^q$-regularity can be used, while for $\beta<1$ other methods must be applied. 

% and for the practically relevant case $\beta<1$ other methods must be applied. This illustrates the wide range of problems related to geophysical flows as well as the limits of methods based on %the contraction mapping principle as maximal regularity, and also Fujita-Kato scheme. 

On the other hand, the Fujita-Kato method is more flexibel in various  situations compared to the approach presented here. This approach allows to  include for example anisotropic spaces. 
Considering for simplicity the case of pure Neumann boundary conditions, where $A_p v = \Delta v$, we may split $e^{tA_p} = e^{t\Delta_H}\circ e^{t\Delta_z}$ into commuting semigroups 
generated by $\Delta_H= \partial_x^2 + \partial_y^2$ and $\Delta_z= \partial_z^2$. So, using the anisotropic estimate
\begin{align*}
\norm{F_p(v)}_{L^p(\Omega)} \leq \norm{v}_{H^{1,p}_zH^{1/p,p}_{xy}} \norm{v}_{L^{p}_zH^{1+1/p,p}_{xy}}, 
\end{align*}
and considering quantities of the form
\begin{align*}
K(v)(t)= \sup_{0<s<t}s^{1/2+1/2p} \norm{v(s)}_{H^{1,p}_zH^{1/p,p}_{xy}} \quad \hbox{and} \quad H(v)(t)= \sup_{0<s<t}s^{1/2+1/2p} \norm{v(s)}_{L^{p}_zH^{1+1/p,p}_{xy}}
\end{align*}
we may distribute time weights anisotropically which leads to initial values 
$$v_0 \in H_z^{1/p,p}H^{1/p}_{xy}\cap L_z^{p}H^{2/p}_{xy} \cap \lpso, \quad p\in (1,\infty),$$ 
which for $p=2$ is slightly better than the result presented here since $H^1(\Omega) \subset H_z^{1/2}H_{xy}^{1/2}\cap L_z^{2}H_{xy}^{1}$.


\begin{thebibliography}{40}


	

		
\bibitem{Amann}
H.~Amann.
\newblock {\em Linear and quasilinear parabolic problems. Vol. I Abstract linear theory} 
\newblock Birkh\"auser Boston, Inc., Boston, MA, 1995.
\newblock \doi{10.1007/978-3-0348-9221-6}


\bibitem{Angenent1990_2}
S.~B.~Angenent.
\newblock Nonlinear analytic semiflows. 
\newblock {\em Proc. Roy. Soc. Edinburgh Sect. A}, 115(1--2):91--107, 1990. 
\newblock \doi{10.1017/S03082105000245988}
	
	
\bibitem{Angenent1990}
S.~B.~Angenent.
\newblock Parabolic equations for curves on surfaces. {I}. {C}urves with {$p$}-integrable curvature. 
\newblock {\em Ann. of Math. (2)}, 132(3):451--483, 1990. 
\newblock \doi{10.2307/1971426}
	
\bibitem{Bergh1984}
J.~Bergh.
\newblock A non-linear complex interpolation result. 
\newblock In {\em Interpolation Spaces and Allied Topics in Analysis: Proceedings of the Conference held in Lund, Sweden, August 29 -- September 1, 1983} Springer Berlin Heidelberg:45--47, 1984. 
\newblock \doi{10.1007/BFb0099091}	
	

\bibitem{Can97}
M. Cannone.
\newblock A generalization of a theorem by {K}ato on {N}avier-{S}tokes equations.  
\newblock {\em Rev. Mat. Iberoamericana}, 13(3):515-541, 1997.
\newblock \doi{10.4171/RMI/229}

\bibitem{CaoTiti2007}
Ch.~Cao and E.~Titi.
\newblock Global well--posedness of the three-dimensional viscous primitive equations of large scale ocean and atmosphere dynamics.
\newblock {\em Annals of Mathematics}, 166:245--267, 2007.
\newblock \doi{10.4007/annals.2007.166.245}


\bibitem{Cotietall2015}
M.~Coti Zelati, A.~Huang, I.~Kukavica, R.~Temam and M.~Ziane.
\newblock The primitive equations of the atmosphere in presence of vapour saturation. 
\newblock  {\em Nonlinearity}, 28(3):625--668, 2015. 
\newblock \doi{10.1088/0951-7715/28/3/625}



\bibitem{ChillFasangova2009}
R.~Chill and E.~Fasangova.
\newblock {\em Gradient Systems}. Lecture Notes of the 13th International Internet Seminar.
\newblock Matfyzpress, Prague, 2010.
\newblock \url{https://www.math.tecnico.ulisboa.pt/~czaja/ISEM/13internetseminar200910.pdf}





%\bibitem{DenkHieberPruess2003} 
%R.~Denk, M.~Hieber, and J.~Pr{\"u}ss.
%\newblock {$\cR$}-boundedness, {F}ourier multipliers and problems of elliptic and parabolic type.
%\newblock {\em Mem. Amer. Math. Soc}, 166, 2003.
%\newblock \doi{10.1090/memo/0788}


\bibitem{EscherSimonett2003}
J.~Escher and G.~Simonett.
\newblock Analyticity of solutions to fully nonlinear parabolic evolution equations on symmetric spaces. 
\newblock  {\em J. Evol. Equ.}, 3(4):549--576 , 2003. 
\newblock \doi{10.1007/s00028-003-0093-z}

%\bibitem{Pruess2003}
%J.~Escher, J.~Pr\"uss and G.~Simonett.
%\newblock A new approach to the regularity of solutions for parabolic equations. 
%\newblock In {\em Evolution equations} Lecture Notes in Pure and Appl. Math., Dekker, New York:167--190, 2003. 
%\newblock \doi{10.1007/s00028-003-0093-z}

%\bibitem{FujitaKato1964}
%H.~Fujita and T.~Kato.
%\newblock On the {N}avier-{S}tokes initial value problem. {I}.
%\newblock {\em Arch. Rational Mech. Anal.}, 16:269--315, 1964.
%\newblock \doi{10.1007/BF00276188}


\bibitem{Farwigetall2017}
R.~Farwig, Y.~Giga and P.~Y.~Hsu.
\newblock The {N}avier-{S}tokes equations with initial values in {B}esov spaces of type $B^{-1+3/q}_{q,\infty}$.
\newblock {\em J. Korean Math. Soc.}, 54(5):1483--1504, 2017.
\newblock \doi{10.4134/JKMS.j160529}

\bibitem{Farwigetall2017b}
	R.~Farwig, Y.~Giga and P.~Y.~Hsu.
	\newblock {O}n the continuity of the solutions to the {N}avier-{S}tokes equations with initial data in critical Besov spaces.
	\newblock Preprint, 2017. 

\bibitem{Farwigetall2016}
R.~Farwig, Y.~Giga and P.~Y.~Hsu.
\newblock Initial values for the {N}avier-{S}tokes equations in spaces with weights in time.             
\newblock {\em Funkcial. Ekvac.}, 59(2):199--216, 2016.
\newblock \doi{10.1619/fesi.59.199}


%\bibitem{Farwigetall2015}
%R.~Farwig, H.~Sohr and W.~Varnhorn.
%\newblock Local strong solutions of the nonhomogeneous {N}avier-{S}tokes system with control of the interval of existence.             
%\newblock {\em Topol. Methods Nonlinear Anal.}, 46(2):999--1012, 2016.





\bibitem{GaldiHieberKashiwabara2015}
G.~P.~Galdi, M.~Hieber and T.~Kashiwabara.
\newblock Strong time-periodic solutions to the 3D  Primitive Equations subject to arbitrary large forces.             
\newblock {\em Nonlinearity}, 30(10):3979, 2017.
%\newblock Preprint, \href{https://arxiv.org/abs/1509.02637v1}{arXiv:1509.02637v1}, to appear in {\em Nonlinearity}.
\newblock \doi{10.1088/1361-6544/aa8166}


%\bibitem{GeissertHeckHieberSawada2012} 
%M.~Geissert, H.~Heck, M.~Hieber, and O.~Sawada.
%\newblock Weak {N}eumann implies {S}tokes.
%\newblock {\em J. Reine Angew. Math.}, 669, 75--100, 2012.
%\newblock \doi{10.1515/CRELLE.2011.150}

%\bibitem{Giga81}
%Y.~Giga.
%\newblock Analyticity of the semigroup generated by the Stokes operator on $L_r$-spaces.
%\newblock {\em Math. Z.}, 178:297--329, 1981.
%\newblock \doi{10.1007/BF01214869}


%\bibitem{Gig85}
%Y.~Giga.
%\newblock Domains of fractional powers of the Stokes operator on $L_r$-spaces.
%\newblock {\em Arch. Rational Mech. Anal.}, 89:251-265, 1985.
%\newblock \doi{10.1007/BF00276875}



\bibitem{GigaGriesHusseinHieberKashiwabara2016}
Y.~Giga, M.~Gries, A.~Hussein, M.~Hieber and T.~Kashiwabara.
\newblock Bounded $H^{\infty}$-Calculus for the Hydrostatic Stokes Operator on $L^p$-Spaces and Applications.             
\newblock {\em Proc. Amer. Math. Soc.}, 145(9):3865--3876, 2017.
\newblock \doi{10.1090/proc/13676}

\bibitem{GigaGriesHusseinHieberKashiwabara2017NN}
Y.~Giga, M.~Gries, A.~Hussein, M.~Hieber and T.~Kashiwabara.
\newblock The Primitive Equations in the scaling invariant space $L^{\infty}(L^1)$.             
\newblock Preprint \href{https://arxiv.org/abs/1710.04434}{arXiv:1710.04434}, 2017.

\bibitem{GigaGriesHusseinHieberKashiwabara2017DN}
Y.~Giga, M.~Gries, A.~Hussein, M.~Hieber and T.~Kashiwabara.
\newblock Dirichlet Neumann.             
\newblock Preprint 2017. 


\bibitem{GigaMiyakawa1985}
Y.~Giga and T.~Miyakawa.
\newblock Solutions in $L_r$ of the Navier-Stokes initial value problem.
\newblock {\em Arch. Ration. Mech. Anal.}, 89(3):267--281, 1985.
\newblock \doi{10.1007/BF00276875}





\bibitem{GigaSawada2003}
Y.~Giga and O.~Sawada.
\newblock On regularizing-decay rate estimates for solutions to the Navier-Stokes initial value problem.
\newblock In {\em Nonlinear analysis and applications: to V. Lakshmikantham on his 80th birthday. Vol. 1, 2}, Kluwer Acad. Publ., Dordrecht, 549--562, 2003.



%\bibitem{GS89}
%Y.~Giga and H. Sohr. 
%\newblock On the Stokes operator in exterior domains.
%\newblock {\em J. Fac. Sci. Univ. Tokyo Sect IA Math.}, 36:103-130, 1989.



%\bibitem{GS91}
%Y.~Giga and H. Sohr. 
%\newblock Abstract $L^p$-estimates for the Cauchy problem with applications to the Navier-Stokes %equations in exterior domains.
%\newblock {\em J. Funct. Anal.}, 102:72-94, 1991.
%\newblock \doi{10.1016/0022-1236(91)90136-S}


\bibitem{Guillenetall2001}
F.~Guill\'en-Gonz\'alez, N.~Masmoudi and M.~A.~Rodr\'\i guez-Bellido.
\newblock Anisotropic estimates and strong solutions of the primitive
equations.
\newblock {\em Differential Integral Equations}, 14(11):1381--1408, 2001.
\newblock \url{http://projecteuclid.org/euclid.die/1356123030}.



\bibitem{HieberKashiwabara2015}
M.~Hieber and T.~Kashiwabara.
\newblock Global Strong Well--Posedness of the Three Dimensional Primitive Equations in $L^p$--Spaces.             
\newblock {\em Arch. Rational Mech. Anal.}, 221(3): 1077--1115, 2016.   
\newblock \doi{10.1007/s00205-016-0979-x}


\bibitem{HieberKashiwabaraHussein2016}
M.~Hieber, T.~Kashiwabara and A.~Hussein.
\newblock Global strong {$L^p$} well-posedness of the 3{D} primitive
equations with heat and salinity diffusion 
 \newblock {\em J. Differential Equations}, 261(12): 6950--6981, 2016.       
\newblock \doi{10.1016/j.jde.2016.09.010}



%\bibitem{HieberSaal2016}
%M.~Hieber and J.~Saal.
%\newblock The Stokes Equation in the $L^p$ Setting: Well-Posedness and Regularity Properties.           
%\newblock Handbook of Mathematical Analysis in Mechanics of Viscous Fluids, submitted.











\bibitem{Hitmeieretall2016}
S.~Hittmeier, R.~Klein, J.~Li, E.~Titi.
\newblock Global well-posedness for passivly transported
nonlinear moisture dynamics with phase changes. 
\newblock {\em Nonlinearity}, 30(10):3676, 2017.
\newblock \doi{10.1088/1361-6544/aa82f1}



%\bibitem{Kobelkov2007}
%G.~M.~Kobelkov.
%\newblock Existence of a solution ``in the large'' for ocean dynamics equations.
%\newblock {\em J. Math. Fluid Mech.}, 9(4):588--610, 2007.
%\newblock \doi{10.1007/s00021-006-0228-4}



\bibitem{KoehnePruessWilke2010}
M.~K\"ohne, J.~Pr\"uss and M.~Wilke.
\newblock On quasilinear parabolic evolution equations in weighted
{$L_p$}-spaces.
\newblock {\em J. Evol. Equ.}, 10(2):443--463, 2010.
\newblock \doi{10.1007/s00028-010-0056-0}.


%\bibitem{Kukavicaetall2014}
%I.~Kukavica, Y.~Pei, W.~Rusin and M.~Ziane.
%\newblock Primitive equations with continuous initial data.
%\newblock {\em Nonlinearity}, 27:1135--1155, 2014.
%\newblock \doi{10.1088/0951-7715/27/6/1135}


%\bibitem{Ziane2007}
%I.~Kukavica and M.~Ziane.
%\newblock On the regularity of the primitive equations of the ocean.
%\newblock {\em Nonlinearity}, 20(12):2739--2753, 2007.
%\newblock \doi{10.1088/0951-7715/20/12/001}



\bibitem{LeCronePruessWilke2014}
J.~LeCrone, J.~Pr\"uss and M.~Wilke.
\newblock On quasilinear parabolic evolution equations in weighted
{$L_p$}-spaces {II}.
\newblock {\em J. Evol. Equ.}, 14(3):509--533, 2014.
\newblock \doi{10.1007/s00028-014-0226-6}.

\bibitem{Lionsetall1992}
J.~L.~Lions, R.~Temam and Sh.~H.~Wang.
\newblock New formulations of the primitive equations of atmosphere and applications.
\newblock {\em Nonlinearity}, 5(2):237--288, 1992.
\newblock \url{http://stacks.iop.org/0951-7715/5/237}

\bibitem{Lionsetall1992_b}
J.~L.~Lions, R.~Temam and Sh.~H.~Wang.
\newblock On the equations of the large-scale ocean.
\newblock {\em Nonlinearity}, 5(5):1007--1053, 1992.
\newblock \url{http://stacks.iop.org/0951-7715/5/1007}

\bibitem{Lionsetall1993}
J.~L.~Lions, R.~Temam and Sh.~H.~Wang.
\newblock Models for the coupled atmosphere and ocean. ({CAO} {I},{II}).
\newblock {\em Comput. Mech. Adv.}, 1:3--119, 1993.

%\bibitem{LiTiti2016}
%J.~Li and E.~Titi.
%\newblock Recent Advances Concerning Certain Class of Geophysical Flows.
%\newblock Preprint %\href{https://arxiv.org/abs/1604.01695}{arXiv:1604.01695}, 2016.

\bibitem{LiTiti2016}
J.~Li and E.~Titi.
\newblock Recent Advances Concerning Certain Class of Geophysical Flows.
\newblock In {\em Handbook of Mathematical Analysis in Mechanics of Viscous Fluids}, Springer International Publishing, 2016.
\newblock \doi{10.1007/978-3-319-10151-4_22-1}



\bibitem{LiTiti2015}
J.~Li and E.~Titi.
\newblock  Existence and uniqueness of weak solutions to viscous primitive equations for certain class of discontinuous initial data.
\newblock {\em SIAM J. Math. Anal.}, 49(1):1--28, 2017.
\newblock \doi{10.1137/15M1050513}




%\bibitem{Lunardi2009}
%A.~Lunardi.
%\newblock {\em Interpolation theory}.
%\newblock Edizioni della Normale, Pisa, 2009.



%\bibitem{Majda2003}
%A.~J.~Majda.
%\newblock {\em Introduction to PDEs and Waves for the Atmosphere and Ocean}. (Courant Lecture Notes in Mathematics vol 9).
%\newblock Providence, RI: American Mathematical Society, 2003.


\bibitem{Masuda1980}
K.~Masuda.
\newblock On the regularity of solutions of the nonstationary {N}avier-{S}tokes equations. 
\newblock In {\em Approximation methods for {N}avier-{S}tokes problems ({P}roc.{S}ympos., {U}niv. {P}aderborn, {P}aderborn,} Lecture Notes in Math., Springer, Berlin, 771:360--370, 1980. 	
\newblock \doi{10.1007/BFb0086917}




%\bibitem{Nau2012}
%T.~Nau.
%\newblock {\em $L^p$-Theory of Cylindrical Boundary Value Problems}. An Operator-Valued Fourier %Multiplier and Functional Calculus Approach. 
%\newblock PhD Thesis, University Konstanz, 2012.
%\newblock \doi{10.1007/978-3-8348-2505-6}

%\bibitem{Pedlosky1987}
%J.~Pedlosky.
%\newblock {\em Geophysical Fluid Dynamics, 2nd edition,}.
%\newblock Springer, New York, 1987.



%\bibitem{Ziane2009}
%M.~Petcu, R.~Temam and M.~Ziane.
%\newblock Some mathematical problems in geophysical fluid dynamics.
%\newblock In {\em Handbook of numerical analysis. {V}ol. {XIV}. {S}pecial volume: computational methods for the atmosphere and the oceans}, 14:577--750, 2009.
%\newblock \doi{10.1016/S1570-8659(08)00212-3}


\bibitem{Pruess2002}
J. Pr\"uss.
\newblock Maximal regularity for evolution equations in {$L_p$}-spaces. 
\newblock {\em Conferenze del Seminario di Matematica dell'Universit\`a di Bari}, 285:1--39, 2002. 


\bibitem{PruessSimonett2016}
J.~Pr\"uss and G.~Simonett.
\newblock {\em Moving interfaces and quasilinear parabolic evolution equations}.
\newblock Birkh\"auser/Springer, [Cham], 2016.
\newblock \doi{10.1007/978-3-319-27698-4}

\bibitem{PruessSimonett2004}
J.~Pr\"uss and G.~Simonett.
\newblock Maximal regularity for evolution equations in weighted {$L_p$}-spaces.
\newblock {\em Arch. Math. (Basel)}, 82(5):415--431, 2004.
\newblock \doi{10.1007/s00013-004-0585-2}.

\bibitem{PruessSimonettWilke2017}
J.~Pr\"uss, G.~Simonett and M.~Wilke.
\newblock Critical spaces for quasilinear parabolic evolution equations and applications. 
\newblock Preprint  \href{https://arxiv.org/abs/1708.08550}{arXiv:1708.08550}, 2017.

%\bibitem{PS93}
%J. Pr\"uss and H. Sohr.
%\newblock Imaginary powers of elliptic second order differential operators in $L^p$-spaces. 
%\newblock {\em Hiroshima Math. J.}, 23:161-192, 1993. 
%\newblock \url{http://projecteuclid.org/euclid.hmj/1206128381}

\bibitem{PruessWilke2016}
J.~Pr\"uss and M.~Wilke.
\newblock Addendum to the Paper ''On quasilinear parabolic evolution equations in weighted {$L_p$}-spaces {II}''.
\newblock {\em J. Evol. Equ.}, published electronically, 2017.
\newblock \doi{10.1007/s00028-017-0382-6}.




%\bibitem{Solonnikov77}
%V.~A.~Solonnikov.
%\newblock Estimates for solutions of nonstationary Navier-Stokes equations.
%\newblock {\em J. Sov. Math.}, 8:467--529, 1977.



%\bibitem{TachimMedjo2010}
%T.~Tachim Medjo.
%\newblock On the uniqueness of {$z$}-weak solutions of the
%three-dimensional primitive equations of the ocean.
%\newblock {\em Nonlinear Anal. Real World Appl.}, 11(3):1413--1421, 2010.
%\newblock \doi{10.1016/j.nonrwa.2009.02.031}



%\bibitem{Tanabe}
%H.~Tanabe.
%\newblock {\em Equations of evolution} 
%\newblock Pitman (Advanced Publishing Program), Boston, Mass.-London, 1979.



\bibitem{Triebel}
H.~Triebel.
\newblock {\em Theory of Function Spaces}.
\newblock (Reprint of 1983 edition) Springer AG, Basel, 2010.





\bibitem{Triebel1978}
H.~Triebel.
\newblock {\em Interpolation theory, function spaces, differential operators.} Second Edition
\newblock Johann Ambrosius Barth, Heidelberg 1978.



%\bibitem{Vallis2006}
%G.~K.~Vallis.
%\newblock {\em Atmospheric and Oceanic Fluid Dynamics.} Second Edition
%\newblock Cambridge Univ. Press, 2006.

%\bibitem{WashingtonParkinson1986}
%W.~M.~Washington and C.~L.~Parkinson.
%\newblock {\em An Introduction to Three Dimensional Climate Modeling.} Second Edition
%\newblock Oxford University Press, Oxford, 1986.

%\bibitem{Whittlesey1965}
%E.~F.~Whittlesey.
%\newblock Analytic functions in {B}anach spaces. 
%\newblock {\em Proc. Amer. Math. Soc.}, 16:1077--1083, 1965. 	
%\newblock \doi{10.2307/2035620}


%\bibitem{Yagi}
%A.~Yagi.
%\newblock {\em Abstract parabolic evolution equations and their applications} 
%\newblock Springer-Verlag, Berlin, 2010.

\end{thebibliography}
\end{document}